\documentclass[10pt,reqno]{amsart}

\usepackage[T1]{fontenc}
\usepackage[utf8]{inputenc}
\usepackage{lmodern}
\usepackage{microtype}

\usepackage{amsmath,amssymb,amsthm,mathtools}
\usepackage{mathrsfs}
\usepackage{enumitem}
\usepackage{xcolor}
\usepackage{hyperref}
\usepackage{tikz-cd}
\usepackage[nameinlink,capitalise,noabbrev]{cleveref}
\usepackage[most]{tcolorbox}

\hypersetup{
    colorlinks=true,
    linkcolor=blue,
    citecolor=blue,
    urlcolor=blue,
    pdftitle={Proper Homotopy Nonrigidity of Open Contractible Manifolds},
    pdfauthor={Donghan Kim}
}

\numberwithin{equation}{section}

\newtheorem{maintheorem}{Theorem}

\newtheorem{theorem}{Theorem}
\newtheorem{proposition}{Proposition}
\newtheorem{lemma}{Lemma}
\newtheorem{corollary}{Corollary}

\theoremstyle{definition}
\newtheorem{definition}{Definition}

\newtheorem{problem}{Open Problem}

\theoremstyle{remark}
\newtheorem{remark}{Remark}

\crefname{maintheorem}{Theorem}{Theorems}
\Crefname{maintheorem}{Theorem}{Theorems}
\crefname{problem}{open problem}{open problems}
\Crefname{problem}{Open Problem}{Open Problems}

\newcommand{\Z}{\mathbb Z}
\newcommand{\Q}{\mathbb Q}

\newcommand{\C}{\mathbb C}

\newcommand{\TOP}{\mathrm{TOP}}
\newcommand{\STOP}{\mathrm{STOP}}
\newcommand{\PL}{\mathrm{PL}}
\newcommand{\DIFF}{\mathrm{DIFF}}

\newcommand{\cS}{\mathcal S}
\newcommand{\cN}{\mathcal N}
\newcommand{\Wh}{\operatorname{Wh}}
\newcommand{\Aut}{\operatorname{Aut}}
\newcommand{\Out}{\operatorname{Out}}
\newcommand{\Int}{\operatorname{int}}
\newcommand{\im}{\operatorname{im}}

\newcommand{\id}{\operatorname{id}}

\newcommand{\simeqp}{\simeq_{\mathrm{prop}}}

\newcommand{\rhot}{\widetilde{\rho}}
\newcommand{\Susp}{\Sigma}
\newcommand{\SL}{\operatorname{SL}}
\newcommand{\PSL}{\operatorname{PSL}}
\newcommand{\cR}{\mathcal R}
\newcommand{\rhoabs}{\rho_{\mathrm{abs}}}

\title{Proper Homotopy Nonrigidity of Open Contractible Manifolds}

\author{Donghan Kim}

\address{
Department of Mathematical Sciences,
Korea Advanced Institute of Science and Technology,
Daejeon, Republic of Korea
}



\date{}

\begin{document}

\begin{abstract}
We study whether proper homotopy equivalence classifies open contractible manifolds arising as interiors of compact contractible manifolds. In dimensions three and four, it does: any two such interiors that are properly homotopy equivalent are homeomorphic. In every even dimension $N\geq 6$ and every odd dimension $N\geq 9$, however, it does not: a single proper homotopy type contains infinitely many pairwise nonhomeomorphic smooth open contractible $N$-manifolds. The cases of dimensions five and seven remain unresolved.
\end{abstract}

\maketitle

\section{Introduction}

Every contractible space has the homotopy type of a point, so ordinary
homotopy theory provides no classification of open contractible
manifolds.  Their topology at infinity, however, can be highly
nontrivial, as illustrated by the Whitehead manifold, interiors of
compact contractible manifolds, and the contractible manifolds arising
from the Davis reflection-group construction; see
\cite{DavisReflection,GuilbaultEnds}.  Proper homotopy theory retains
part of this missing information through the pro-homotopy type of
neighborhoods of infinity.

A classical theorem of Stallings~\cite{StallingsEuclidean} states that
a contractible open \(n\)-manifold, \(n\geq5\), is homeomorphic to
\(\mathbb R^n\) if and only if it is simply connected at infinity.
Consequently,
\[
M\simeqp\mathbb R^n
\quad\Longrightarrow\quad
M\cong_{\TOP}\mathbb R^n.
\]
Thus the Euclidean proper homotopy type is topologically rigid.  This
raises the natural question of whether proper homotopy equivalence
classifies other open contractible manifolds up to homeomorphism.

We study this question in the canonical tame subclass consisting of
interiors of compact contractible manifolds.  This class contains the
Euclidean model \(\mathbb R^n=\Int(B^n)\), but it also permits
nontrivial topology at infinity: the collared end records the boundary
of the compactification.  The resulting dimensional picture is nearly
complete.

\begin{maintheorem}[Dimensional rigidity and nonrigidity]
\label{thm:main-classification}
Among interiors of compact contractible manifolds, the following hold.
\begin{enumerate}[label=\textup{(\roman*)}]
\item In dimensions \(N=3,4\), proper homotopy equivalence determines
the topological homeomorphism type.  In fact, the corresponding compact
contractible manifolds are homeomorphic.
\item In every even dimension \(N\geq6\) and every odd dimension
\(N\geq9\), one proper homotopy type contains infinitely many pairwise
nonhomeomorphic smooth open contractible \(N\)-manifolds that are
interiors of compact contractible smooth manifolds.
\end{enumerate}
Dimensions five and seven remain unresolved.
\end{maintheorem}

The organizing idea is a boundary--interior dictionary.  If \(C\) is
compact and contractible with connected boundary, then
\[
\Int(C)\simeqp c^\circ(\partial C).
\]
For compact contractible manifolds \(C_0\) and \(C_1\), this gives
\[
\Int(C_0)\simeqp\Int(C_1)
\quad\Longleftrightarrow\quad
\partial C_0\simeq\partial C_1.
\]
In the other direction, a homeomorphism of interiors forces the
boundaries to be topologically \(h\)-cobordant.  These facts yield the
following reusable principle, which is the common geometric step in
all of the nonrigidity constructions.

\begin{theorem}[Boundary-to-interior nonrigidity principle]
\label{thm:boundary-to-interior-transfer}
Let \(N\geq2\), and let \(\{C_j\}_{j\in J}\) be compact contractible
topological \(N\)-manifolds with connected boundaries.  Suppose that
the boundaries \(\partial C_j\) are mutually homotopy equivalent and
pairwise not topologically \(h\)-cobordant.  Then the interiors
\(\Int(C_j)\) are mutually properly homotopy equivalent and pairwise
nonhomeomorphic.
\end{theorem}

The high-dimensional constructions supplying the boundary families
are summarized by the next two results.

\begin{maintheorem}[Even-dimensional nonrigidity]
\label{thm:even-main}
For every even integer \(N\geq6\), there exists an infinite family
\(\{C_j\}_{j\in J}\) of compact contractible smooth \(N\)-manifolds
whose interiors are mutually properly homotopy equivalent and pairwise
nonhomeomorphic.  Their boundaries may be chosen to be homotopy
equivalent integral homology \((N-1)\)-spheres that are pairwise not
topologically \(h\)-cobordant.
\end{maintheorem}

\begin{maintheorem}[Odd-dimensional nonrigidity]
\label{thm:odd-main}
For every odd integer \(N\geq9\), there exists an infinite family of
compact contractible smooth \(N\)-manifolds \(\{C_j\}_{j\in J}\) such
that the interiors \(M_j=\Int(C_j)\) are all properly homotopy
equivalent but pairwise nonhomeomorphic.  Their boundaries may be
chosen to be homotopy equivalent integral homology \((N-1)\)-spheres
that are pairwise not topologically \(h\)-cobordant.
\end{maintheorem}

The low-dimensional part of \cref{thm:main-classification} follows
from the same boundary dictionary.  In dimension three one uses the
Poincar\'e and generalized Schoenflies theorems.  In dimension four,
rigidity of integral homology three-spheres and Boyer's classification
of simply connected topological four-manifolds with prescribed
boundary identify the compact fillings~\cite{BoyerBoundary}.

For even \(N=2d\), a general representation-theoretic criterion turns
a nonzero reduced multisignature direction in \(L_N^s(\Z\pi)\) into an
infinite affine family of rho invariants.  Kervaire realization and
filling then supply the required boundaries and interiors.  The single
finite superperfect group \(\pi=\SL(2,7)\) satisfies the criterion in
every even dimension \(N\geq6\).

The odd-dimensional argument is organized around a second reusable
criterion involving finite \(\Out(\pi)\), antisimple realization, and
Weinberger's absolute higher rho invariant.  For odd \(N\geq9\), we
verify it for
\[
\pi=\Gamma\times\SL(2,7),
\]
where \(\Gamma\) is the fundamental group of a closed hyperbolic
integral homology three-sphere.  An equivariant Chern-character
calculation moves a finite-group multisignature direction into the
required odd-dimensional \(L\)-group, and the absolute higher rho
invariant separates infinitely many topological \(h\)-cobordism
classes~\cite{WeinbergerRho,DFW}.

The paper is organized as follows.  We first review proper homotopy and
prove the boundary--interior transfer principle, then establish
rigidity in dimensions three and four.  A detailed six-dimensional
prototype is followed by the general even-dimensional criterion.  We
next prove and verify the odd-dimensional criterion, explain the
remaining obstruction in dimension seven, and conclude with the two
open problems in dimension five.

\section{Proper homotopy and topology at infinity}
\label{sec:proper-preliminaries}

\subsection{Proper homotopy equivalence}

A continuous map
\[
f\colon X\longrightarrow Y
\]
between locally compact Hausdorff spaces is called \emph{proper} if
\(f^{-1}(K)\) is compact for every compact subset \(K\subseteq Y\).

A proper map \(f\colon X\to Y\) is a \emph{proper homotopy
equivalence} if there is a proper map \(g\colon Y\to X\) and proper
homotopies
\[
g\circ f\simeq_{\mathrm{prop}}\operatorname{id}_X,
\qquad
f\circ g\simeq_{\mathrm{prop}}\operatorname{id}_Y.
\]

By forgetting the properness condition, every proper homotopy
equivalence is an ordinary homotopy equivalence.  The converse is
false in general: ordinary homotopies can move points from arbitrarily
far out toward a fixed compact subset and therefore need not preserve
the topology at infinity.

\subsection{The fundamental group at infinity}

Let \(X\) be a one-ended, locally compact, locally path-connected
space, and choose a nested cofinal sequence of connected neighborhoods
of infinity
\[
U_0\supseteq U_1\supseteq U_2\supseteq\cdots
\]
together with an appropriate proper base ray.  The fundamental group
at infinity is represented by the inverse sequence
\[
\pi_1(U_0)
\longleftarrow
\pi_1(U_1)
\longleftarrow
\pi_1(U_2)
\longleftarrow\cdots,
\]
considered up to pro-isomorphism.

The space \(X\) is called \emph{simply connected at infinity} if this
inverse sequence is pro-trivial.  Equivalently, for every compact set
\(K\subseteq X\), there is a compact set \(L\supseteq K\) such that
every loop in \(X-L\) is null-homotopic in \(X-K\).

Proper homotopy equivalences preserve the number of ends.  They also
preserve the pro-isomorphism class of the fundamental group at
infinity; see \cite{GuilbaultEnds}.  Consequently, simple connectivity
at infinity is a proper homotopy invariant.

\subsection{Rigidity of the Euclidean proper homotopy type}

We record the classical Euclidean rigidity consequence that motivates
the main question of this paper.

\begin{theorem}[{\cite[Theorem~4]{StallingsEuclidean}}]
\label{thm:stallings-euclidean}
Let \(n\geq5\), and let \(M^n\) be a contractible smooth open
\(n\)-manifold.  Then \(M\) is PL-homeomorphic, and hence
topologically homeomorphic, to \(\mathbb R^n\) if and only if \(M\) is
simply connected at infinity.
\end{theorem}


\begin{corollary}[Proper homotopy rigidity of Euclidean space]
\label{cor:proper-rigidity-euclidean}
Let \(n\geq5\), and let \(M^n\) be a smooth open \(n\)-manifold.  If
\[
M\simeq_{\mathrm{prop}}\mathbb R^n,
\]
then
\[
M\cong_{\TOP}\mathbb R^n.
\]
\end{corollary}

\begin{proof}
A proper homotopy equivalence is, after forgetting properness, an
ordinary homotopy equivalence.  Since \(\mathbb R^n\) is contractible,
\(M\) is contractible.

Moreover, proper homotopy equivalence preserves the pro-fundamental
group at infinity.  Since \(\mathbb R^n\) is simply connected at
infinity, \(M\) is simply connected at infinity.  The conclusion now
follows from \cref{thm:stallings-euclidean}.
\end{proof}


\section{The boundary-to-interior transfer principle}

\subsection{Proper maps and open cones}

Let \(Y\) be a compact space.  Its open cone is
\[
c^\circ Y
=
Y\times[0,\infty)\big/
\bigl(Y\times\{0\}\bigr).
\]
We write the cone point as \(v\).

\begin{lemma}\label{lem:cone-map}
Let \(Y\) and \(Z\) be compact Hausdorff spaces.  If
\[
f\colon Y\longrightarrow Z
\]
is a homotopy equivalence, then
\[
c^\circ f\colon c^\circ Y\longrightarrow c^\circ Z,
\qquad
[y,t]\longmapsto[f(y),t],
\]
is a proper homotopy equivalence.
\end{lemma}

\begin{proof}
Since the radial coordinate is unchanged, \(c^\circ f\) is proper.
If \(g\colon Z\to Y\) is a homotopy inverse, then \(c^\circ g\) is
also proper.

A homotopy \(H\colon Y\times I\to Y\) induces a proper homotopy
\[
c^\circ H([y,t],s)=[H(y,s),t],
\]
again because the radial coordinate is fixed.  Applying this to the
homotopies
\[
gf\simeq\id_Y,
\qquad
fg\simeq\id_Z
\]
proves the assertion.
\end{proof}

\subsection{Interiors of compact contractible manifolds}

\begin{lemma}\label{lem:interior-cone}
Let \(C\) be a compact contractible manifold with nonempty connected
boundary \(Y=\partial C\).  Then
\[
\Int(C)\simeqp c^\circ Y.
\]
\end{lemma}

\begin{proof}
Choose a collar of the boundary and reparametrize its open part so
that
\[
\Int(C)=K\cup_Y\bigl(Y\times[0,\infty)\bigr),
\]
where \(K\) is a compact codimension-zero submanifold, the copy of
\(Y\) is identified with \(Y\times\{0\}\), and the inclusion
\(K\hookrightarrow C\) is a homotopy equivalence.  Hence \(K\) is
contractible.

Let
\[
q\colon \Int(C)\longrightarrow \Int(C)/K
\]
be the quotient map collapsing \(K\) to a point.  The quotient is
canonically homeomorphic to the open cone:
\[
\Int(C)/K\cong c^\circ Y.
\]
The inclusion \(K\hookrightarrow\Int(C)\) is a cofibration.  Since
\(K\) is contractible, it follows that \(q\) is a homotopy equivalence.

We now verify the proper statement.  Put
\[
T=Y\times[1,\infty)\subset\Int(C).
\]
Under the above identification of the quotient with \(c^\circ Y\),
the map \(q\) restricts to the identity on \(T\).  Both inclusions
\[
T\hookrightarrow\Int(C)
\qquad\text{and}\qquad
T\hookrightarrow c^\circ Y
\]
are cofibrations.  Since \(q\) is a homotopy equivalence and
\(q|_T=\id_T\), the relative homotopy-inverse theorem for
cofibrations, applied to
\[
q\colon(\Int(C),T)\longrightarrow(c^\circ Y,T),
\]
gives a homotopy inverse
\[
r\colon c^\circ Y\longrightarrow\Int(C)
\]
such that \(r|_T=\id_T\), together with homotopies
\[
rq\simeq\id_{\Int(C)}
\qquad\text{and}\qquad
qr\simeq\id_{c^\circ Y}
\]
which are stationary on \(T\).

The map \(q\) is proper because it collapses only the compact set
\(K\) and preserves the radial coordinate off \(K\).  The closures
of the complements of \(T\) in both spaces are compact.  Since \(r\)
is the identity on \(T\), the inverse image under \(r\) of a compact
set is a closed subset of the union of a compact radial core and a
compact subset of \(T\), and is therefore compact.  The inverse
homotopies are stationary on \(T\); the same compact-core argument
shows that they are proper.  Thus
\[
\Int(C)\simeqp c^\circ Y.
\]
\end{proof}

\begin{corollary}\label{cor:boundary-homotopy-proper}
Let \(C_0\) and \(C_1\) be compact contractible \(n\)-manifolds,
where \(n\geq2\).  If
\[
\partial C_0\simeq\partial C_1,
\]
then
\[
\Int(C_0)\simeqp\Int(C_1).
\]
\end{corollary}

\begin{proof}
Poincar\'e--Lefschetz duality shows that each boundary is an integral
homology \((n-1)\)-sphere, and in particular is nonempty and connected.
Combine \cref{lem:interior-cone,lem:cone-map}.
\end{proof}

\begin{proposition}[Boundary recovery from the proper homotopy type]
\label{prop:boundary-recovery}
Let \(Y\) and \(Z\) be compact ANRs.  If
\[
c^\circ Y\simeqp c^\circ Z,
\]
then
\[
Y\simeq Z.
\]
\end{proposition}

\begin{proof}
For \(t>0\), let
\[
U_t(Y)=Y\times[t,\infty)\subset c^\circ Y.
\]
The neighborhoods \(U_t(Y)\) are cofinal at infinity, and each of
them deformation retracts onto \(Y\times\{t\}\).  Moreover, every
bonding inclusion
\[
U_{t'}(Y)\hookrightarrow U_t(Y),
\qquad t'\geq t,
\]
is a homotopy equivalence.  Thus the pro-homotopy type at infinity of
\(c^\circ Y\) is represented by the constant inverse system on
\(Y\).  The analogous statement holds for \(c^\circ Z\).

To make the proper-homotopy step explicit, let \(F:X\to X'\) be a
proper map and choose compact exhaustions \(\{K_i\}\) and \(\{L_j\}\)
of its source and target.  For every \(j\), the compact set
\(F^{-1}(L_j)\) is contained in some \(K_i\), and hence
\[
F(X-K_i)\subseteq X'-L_j.
\]
These restrictions define a morphism of the inverse systems of
neighborhoods of infinity.  If \(H:X\times I\to X'\) is a proper
homotopy, then \(H^{-1}(L_j)\) is compact and so is its projection to
\(X\).  For sufficiently large \(i\), therefore,
\[
H((X-K_i)\times I)\subseteq X'-L_j.
\]
Thus properly homotopic maps induce the same pro-morphism, and a proper
homotopy inverse induces an inverse pro-morphism; compare
\cite[\S3.3.2 and the proof of Proposition~3.4.12]{GuilbaultEnds}.
Consequently the constant pro-objects determined by \(Y\) and \(Z\)
are isomorphic in the pro-homotopy category of ANRs.

For completeness, the constant-object functor is fully faithful
directly from the defining morphism formula for a pro-category:
\[
\operatorname{Hom}_{\operatorname{Pro}(\operatorname{Ho}(\mathrm{ANR}))}
  (cY,cZ)
=
\lim_j\operatorname*{colim}_i
\operatorname{Hom}_{\operatorname{Ho}(\mathrm{ANR})}(Y,Z)
=
[Y,Z].
\]
Consequently, there are homotopy classes
\[
[f]\colon Y\longrightarrow Z,
\qquad
[g]\colon Z\longrightarrow Y
\]
whose composites are the identity homotopy classes.  Therefore
\(Y\simeq Z\).
\end{proof}

\begin{corollary}[Boundary detection for compact contractible
interiors]\label{cor:boundary-detection}
Let \(C_0\) and \(C_1\) be compact contractible \(n\)-manifolds,
where \(n\geq2\).  Then
\[
\Int(C_0)\simeqp\Int(C_1)
\quad\Longleftrightarrow\quad
\partial C_0\simeq\partial C_1.
\]
\end{corollary}

\begin{proof}
Both boundaries are compact ANRs because they are compact manifolds,
and Poincar\'e--Lefschetz duality shows that they are connected.  The
reverse implication is
\cref{cor:boundary-homotopy-proper}.  For
the forward implication, use \cref{lem:interior-cone} to obtain
\[
c^\circ(\partial C_0)
\simeqp
c^\circ(\partial C_1),
\]
and then apply \cref{prop:boundary-recovery}.
\end{proof}

\subsection{Uniqueness of manifold completion}

The following completion statement is an elementary nested-collar
form of the uniqueness of open cone neighborhoods; compare
Kwun~\cite{KwunOpenCone}.  We include the argument because the
dimension-free form will also be used in the discussion of dimension
five.

\begin{proposition}[Boundary uniqueness up to \(h\)-cobordism]
\label{prop:completion}
Let \(C_0\) and \(C_1\) be compact connected topological
\(n\)-manifolds with connected boundary, where \(n\geq2\).  If
\[
\Int(C_0)\cong_{\TOP}\Int(C_1),
\]
then \(\partial C_0\) and \(\partial C_1\) are topologically
\(h\)-cobordant.
\end{proposition}

\begin{proof}
Use the given homeomorphism to identify \(\Int(C_0)\) and
\(\Int(C_1)\) with one open manifold \(M\).  For \(i=0,1\), choose a
proper topological embedding onto a closed product neighborhood of
infinity
\[
\lambda_i\colon \partial C_i\times[0,\infty)\longrightarrow M
\]
and put
\[
U_i(t)=\lambda_i\bigl(\partial C_i\times[t,\infty)\bigr).
\]
The sets \(U_i(t)\) are cofinal among neighborhoods of infinity.  For
\(t'\geq t\), the inclusion \(U_i(t')\hookrightarrow U_i(t)\) is a
homotopy equivalence.

Choose \(a,b,c,d\geq0\), with \(c\geq a\) and \(d\geq b\), by
cofinality so that
\[
\begin{aligned}
A&:=U_0(a),&
B&:=U_1(b)\subset\Int(A),\\
C&:=U_0(c)\subset\Int(B),&
D&:=U_1(d)\subset\Int(C).
\end{aligned}
\]
Let
\[
p\colon B\hookrightarrow A,\qquad
q\colon C\hookrightarrow B,\qquad
r\colon D\hookrightarrow C
\]
be the inclusions.  The composites \(p\circ q\colon C\hookrightarrow
A\) and \(q\circ r\colon D\hookrightarrow B\) are inclusions between
tails of the same respective collars, hence are homotopy equivalences.
In the homotopy category, \(q\) has left inverse
\[
(p\circ q)^{-1}\circ p
\]
and right inverse
\[
r\circ(q\circ r)^{-1}.
\]
Thus \(q\) is a homotopy equivalence, and then so is \(p\).

Set
\[
W=A\setminus\Int(B).
\]
The configuration is shown schematically in
\cref{fig:nested-collars}.

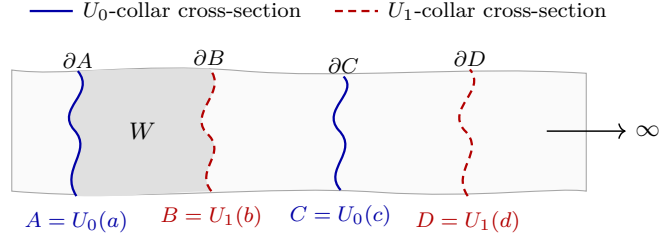
\begin{figure}[t]
\centering
\begin{tikzpicture}[
  x=0.95cm,
  y=0.72cm,
  every node/.style={font=\footnotesize},
  zero collar/.style={draw=blue!65!black,line width=0.85pt},
  one collar/.style={draw=red!70!black,line width=0.85pt,densely dashed},
  boundary label/.style={font=\footnotesize,inner sep=1pt}
]
  \path[fill=black!2,draw=black!35,line width=0.4pt]
    (0,-1.10)
    .. controls (1.8,-1.30) and (3.6,-0.98) .. (5.1,-1.12)
    .. controls (6.5,-1.24) and (7.6,-1.02) .. (8.0,-1.10)
    -- (8.0,1.06)
    .. controls (6.3,1.20) and (4.6,0.94) .. (3.15,1.11)
    .. controls (1.75,1.26) and (0.70,0.99) .. (0,1.12)
    -- cycle;

  \path[fill=black!12,draw=none]
    (0.90,-1.19)
    .. controls (0.66,-0.70) and (1.14,-0.37) .. (0.87,0.02)
    .. controls (0.60,0.42) and (1.18,0.68) .. (0.92,1.10)
    .. controls (1.50,1.15) and (2.20,1.14) .. (2.75,1.15)
    .. controls (3.02,0.73) and (2.47,0.36) .. (2.70,-0.02)
    .. controls (3.04,-0.40) and (2.53,-0.70) .. (2.78,-1.10)
    .. controls (2.15,-1.12) and (1.48,-1.22) .. (0.90,-1.19)
    -- cycle;

  \draw[zero collar]
    (0.90,-1.19)
    .. controls (0.66,-0.70) and (1.14,-0.37) .. (0.87,0.02)
    .. controls (0.60,0.42) and (1.18,0.68) .. (0.92,1.10);
  \draw[one collar]
    (2.78,-1.10)
    .. controls (2.53,-0.70) and (3.04,-0.40) .. (2.70,-0.02)
    .. controls (2.47,0.36) and (3.02,0.73) .. (2.75,1.15);
  \draw[zero collar]
    (4.58,-1.08)
    .. controls (4.32,-0.70) and (4.85,-0.38) .. (4.55,0.02)
    .. controls (4.30,0.40) and (4.84,0.70) .. (4.62,1.00);
  \draw[one collar]
    (6.35,-1.19)
    .. controls (6.10,-0.72) and (6.62,-0.38) .. (6.30,0.02)
    .. controls (6.04,0.43) and (6.62,0.73) .. (6.38,1.13);

  \node[boundary label,above] at (0.92,1.10) {$\partial A$};
  \node[boundary label,above] at (2.75,1.15) {$\partial B$};
  \node[boundary label,above] at (4.62,1.00) {$\partial C$};
  \node[boundary label,above] at (6.38,1.13) {$\partial D$};

  \node[text=blue!65!black,anchor=north,yshift=-2pt] at (0.90,-1.19)
    {$A=U_0(a)$};
  \node[text=red!70!black,anchor=north,yshift=-2pt] at (2.78,-1.10)
    {$B=U_1(b)$};
  \node[text=blue!65!black,anchor=north,yshift=-2pt] at (4.58,-1.08)
    {$C=U_0(c)$};
  \node[text=red!70!black,anchor=north,yshift=-2pt] at (6.35,-1.19)
    {$D=U_1(d)$};

  \node[font=\small] at (1.82,0) {$W$};
  \draw[->,line width=0.55pt] (7.45,0) -- (8.55,0)
    node[right,font=\small] {$\infty$};

  \draw[zero collar] (0.25,2.2) -- (0.85,2.2)
    node[right,text=black] {$U_0$-collar cross-section};
  \draw[one collar] (4.50,2.2) -- (5.10,2.2)
    node[right,text=black] {$U_1$-collar cross-section};
\end{tikzpicture}
\caption{The alternating collar tails used in the proof of
\cref{prop:completion}.  The tail determined by each cross-section lies
to its right, toward infinity, and
\(A\supset B\supset C\supset D\).  The shaded region is
\(W=A\setminus\Int(B)\).  The drawing records only the nesting; it does
not assert that \(W\) is a product cobordism.}
\label{fig:nested-collars}
\end{figure}

The nested collars are bicollared.  Since \(B\) is a closed
neighborhood of infinity, \(M\setminus\Int(B)\) is compact.  Since
\(A\) is closed in \(M\),
\[
W=A\cap\bigl(M\setminus\Int(B)\bigr)
\]
is a closed subset of this compact set and is therefore compact.
Thus \(W\) is a compact topological \(n\)-manifold with
\[
\partial W=\partial A\sqcup\partial B,\qquad
\partial A\cong\partial C_0,\quad
\partial B\cong\partial C_1.
\]
Collapsing the product tail
\(B\cong\partial B\times[0,\infty)\) onto \(\partial B\), while fixing
\(W\), gives a deformation retraction
\[
A=W\cup_{\partial B}B\longrightarrow W.
\]
Hence \(j\colon W\hookrightarrow A\) is a homotopy equivalence.  The
inclusions \(\partial A\hookrightarrow A\) and
\(\partial B\hookrightarrow B\) are homotopy equivalences by the
product-collar retractions.  Since the first factors through \(j\), the
inclusion \(\partial A\hookrightarrow W\) is a homotopy equivalence.
Similarly, the composite
\[
\partial B\hookrightarrow W\xrightarrow{\,j\,}A
\]
agrees with \(\partial B\hookrightarrow B\xrightarrow{\,p\,}A\), which
is a homotopy equivalence.  Thus
\(\partial B\hookrightarrow W\) is a homotopy equivalence as well.
Therefore \(W\) is the required topological \(h\)-cobordism.
\end{proof}

\begin{remark}
The conclusion is \(h\)-cobordism, not necessarily a product
cobordism.  The Whitehead torsion of the resulting \(h\)-cobordism
need not vanish.
\end{remark}

\begin{proof}[Proof of
\cref{thm:boundary-to-interior-transfer}]
Mutual proper homotopy equivalence follows from
\cref{cor:boundary-homotopy-proper}.  If two interiors were
homeomorphic, \cref{prop:completion} would make their boundaries
topologically \(h\)-cobordant, contrary to the hypothesis.  Applying
the argument pairwise proves the family statement.
\end{proof}

\section{Low-dimensional rigidity for compact contractible interiors}
\label{sec:low-dimensional-rigidity}

\begin{theorem}[Low-dimensional rigidity for compact contractible
interiors]\label{thm:low-dimensional-rigidity}
Let \(n\in\{3,4\}\), and let \(C_0\) and \(C_1\) be compact
contractible topological \(n\)-manifolds.  If
\[
\Int(C_0)\simeqp\Int(C_1),
\]
then
\[
C_0\cong_{\TOP}C_1.
\]
In particular,
\[
\Int(C_0)\cong_{\TOP}\Int(C_1).
\]
For \(n=3\), every such compact manifold is homeomorphic to \(B^3\).
\end{theorem}

We now prove \cref{thm:low-dimensional-rigidity}.  The point is that,
for an interior of a compact contractible manifold, the proper
homotopy type recovers the homotopy type of the boundary by
\cref{cor:boundary-detection}.  In dimensions three and four, that
boundary information is sufficiently rigid.

\begin{lemma}[Rigidity of integral homology three-spheres]
\label{lem:homology-three-sphere-rigidity}
Let \(Y_0\) and \(Y_1\) be closed oriented integral homology
three-spheres.  If
\[
Y_0\simeq Y_1,
\]
then
\[
Y_0\cong_{\TOP}Y_1.
\]
\end{lemma}

\begin{proof}
By the prime decomposition theorem, each \(Y_i\) is a connected sum
of irreducible integral homology three-spheres, with no
\(S^1\times S^2\)-summand~\cite{MilnorPrimeDecomposition}.  An
isomorphism of fundamental groups identifies the nontrivial freely
indecomposable factors in these decompositions.

By geometrization, an irreducible factor with infinite fundamental
group is aspherical and has torsion-free fundamental group.
Corresponding factors are \(K(\pi,1)\)-spaces and hence are homotopy
equivalent.  Such three-manifolds are topologically rigid: a homotopy
equivalence is
homotopic to a homeomorphism; see Turaev~\cite{TuraevGeometric} and
Kreck--L\"uck~\cite[Theorem~0.7]{KreckLuckRigidity}.  An irreducible factor with
finite fundamental group is spherical.  The classification of
spherical space forms shows that an integral homology three-sphere of
this type is either \(S^3\) or the Poincar\'e homology sphere; see
\cite{MilnorFreeActions,WolfSpaceForms}.  The Poincar\'e theorem and
geometrization used here follow from Perelman's work
\cite{PerelmanSurgery,PerelmanExtinction}.

Thus corresponding prime factors are homeomorphic.  Uniqueness of
prime decomposition now gives \(Y_0\cong_{\TOP}Y_1\).
\end{proof}

\begin{remark}
The homology-sphere hypothesis in
\cref{lem:homology-three-sphere-rigidity} is essential.  Closed
three-manifolds need not be determined up to homeomorphism by their
ordinary homotopy type; lens spaces give classical counterexamples.
\end{remark}

\begin{proof}[Proof of \cref{thm:low-dimensional-rigidity}]
Suppose first that \(n=3\).  Since \(C_i\) is contractible,
Poincar\'e--Lefschetz duality shows that \(\partial C_i\) is an
integral homology two-sphere, hence is homeomorphic to \(S^2\).
The double
\[
D C_i=C_i\cup_{\partial C_i}C_i
\]
is a simply connected closed three-manifold, so the Poincar\'e theorem
implies \(D C_i\cong S^3\).  The copy of \(\partial C_i\) in the
double is bicollared.  Brown's generalized Schoenflies theorem
\cite{BrownSchoenflies} therefore implies that the closure of each
complementary component is a three-ball.  Hence
\[
C_i\cong B^3.
\]

Now suppose that \(n=4\).  Put \(Y_i=\partial C_i\).  By
Poincar\'e--Lefschetz duality, each \(Y_i\) is an integral homology
three-sphere.  The assumed proper homotopy equivalence and
\cref{cor:boundary-detection} give
\[
Y_0\simeq Y_1.
\]
By \cref{lem:homology-three-sphere-rigidity}, there is a homeomorphism
\(Y_0\cong Y_1\).  After reversing the orientation of \(C_1\), if
necessary, identify both oriented boundaries with a fixed oriented
homology sphere \(Y\).

Both \(C_0\) and \(C_1\) are simply connected, and both have the zero
intersection form.  In Boyer's notation, they therefore determine
elements of \(V_0(Y)\), the set of oriented homeomorphism types of
compact simply connected oriented four-manifolds with boundary \(Y\)
and intersection form zero.  Boyer's classification shows that, when
\(Y\) is an integral homology three-sphere and the prescribed form is
even, this set has exactly one element
\cite[Proposition~0.5]{BoyerBoundary}.  Since
the zero form is even, it follows that
\[
C_0\cong_{\TOP}C_1.
\]
Taking interiors completes the proof.
\end{proof}

\begin{remark}
The conclusion in dimension four is topological.  No corresponding
smooth rigidity statement is asserted.  Moreover,
\cref{thm:low-dimensional-rigidity} concerns only open manifolds that
are interiors of compact contractible manifolds; it makes no claim
about arbitrary open contractible three- or four-manifolds with wild
ends.
\end{remark}

\section{The six-dimensional prototype}
\label{sec:six-dimensional-prototype}

We first give the dimension-six construction as a concrete prototype.
It records the normal-invariant and rho-variation calculations in
detail.  The next section isolates the shorter general criterion and
deduces all even-dimensional cases from it.

\subsection{Homology spheres and the group \texorpdfstring{\(\pi=\SL(2,7)\)}{pi=SL(2,7)}}

\subsubsection{Superperfect groups}

\begin{definition}
A group \(\pi\) is \emph{superperfect} if
\[
H_1(\pi;\Z)=0
\qquad\text{and}\qquad
H_2(\pi;\Z)=0.
\]
\end{definition}

The first condition says that \(\pi\) is perfect, while the second says
that its Schur multiplier vanishes.

\begin{theorem}[{\cite[Theorem 1]{KervaireHomologySpheres}}]\label{thm:kervaire-realization}
Let \(n\geq5\).  A group \(\pi\) occurs as the fundamental group of a
smooth integral homology \(n\)-sphere
if \(\pi\) is finitely
presented and superperfect.
\end{theorem}



\begin{lemma}\label{lem:sl27-superperfect}
The group \(\pi=\SL(2,7)\) is finite, finitely presented, and
superperfect.
\end{lemma}

\begin{proof}
The group \(\PSL(2,7)\) is nonabelian simple, and
\(\SL(2,7)\) is its universal central extension:
\[
1\longrightarrow \Z/2
\longrightarrow \SL(2,7)
\longrightarrow \PSL(2,7)
\longrightarrow1.
\]
A universal central extension of a perfect group is superperfect.
See, for example, the character and Schur multiplier information in
the \emph{Atlas of Finite Groups}~\cite{Atlas} or
Karpilovsky~\cite{Karpilovsky}.
\end{proof}

Combining \cref{thm:kervaire-realization,lem:sl27-superperfect}
gives the following.

\begin{corollary}\label{cor:existence-X}
There exists a closed oriented smooth five-manifold \(X\) such that
\[
H_*(X;\Z)\cong H_*(S^5;\Z)
\qquad\text{and}\qquad
\pi_1(X)\cong\SL(2,7).
\]
\end{corollary}

For the remainder of this section, fix such a manifold \(X\) and write
\[
\pi=\pi_1(X)\cong\SL(2,7).
\]

\subsection{The surgery exact sequence}

All manifolds in this section are oriented.  Thus the orientation
character is trivial, and every group ring \(\Z\pi\) is equipped with
the standard involution \(g\mapsto g^{-1}\).

\subsubsection{The simple structure set}

\begin{definition}
The smooth simple structure set \(\cS^s_{\DIFF}(X)\) consists of
equivalence classes of pairs
\[
(M,f),
\]
where \(M\) is a closed smooth five-manifold and
\[
f\colon M\longrightarrow X
\]
is a simple homotopy equivalence.

Two pairs \((M_0,f_0)\) and \((M_1,f_1)\) are equivalent if there is
an \(s\)-cobordism \(W\) from \(M_0\) to \(M_1\), together with a
simple homotopy equivalence
\[
F\colon W\longrightarrow X\times I
\]
restricting to the given maps on the boundary.
\end{definition}

For \(\dim X=5\), the relevant portion of the smooth surgery exact
sequence is
\begin{equation}\label{eq:surgery-sequence}
\cN_{\partial}(X\times I)
\xrightarrow{\theta}
L_6^s(\Z\pi)
\xrightarrow{\omega}
\cS^s_{\DIFF}(X)
\xrightarrow{\eta}
\cN(X)
\xrightarrow{\theta}
L_5^s(\Z\pi).
\end{equation}
Here
\[
\cN(X)\cong[X,G/O]
\]
and
\[
\cN_{\partial}(X\times I)
\cong
[X\times I,X\times\partial I;G/O].
\]
Here the letter \(G\) in \(G/O\) denotes the stable monoid of
self-homotopy equivalences of spheres; it is unrelated to the
fundamental group \(\pi=\pi_1(X)\).
See Wall~\cite{Wall} and Ranicki~\cite{RanickiAG}.

The map \(\omega\) is Wall realization: an element of
\(L_6^s(\Z\pi)\) acts on the simple structure set.  Exactness means in
particular that
\[
\omega(a)=\omega(b)
\]
at the base structure if and only if
\[
a-b\in\im(\theta).
\]

\subsubsection{Low-dimensional homotopy groups of \texorpdfstring{\(G/O\)}{G/O}}

The low-dimensional homotopy groups of \(G/O\) are
\[
\pi_i(G/O)=
\begin{cases}
\mathbb Z/2,& i=2,6,\\
\mathbb Z,& i=4,\\
0,& i=1,3,5.
\end{cases}
\]
See, for example, Ranicki~\cite[Remark~9.22]{RanickiAG}.

\begin{lemma}\label{lem:normal-invariants-zero}
For the integral homology five-sphere \(X\),
\[
[X,G/O]=0.
\]
\end{lemma}

\begin{proof}
Let \(P_5(G/O)\) denote the fifth Postnikov section of \(G/O\).
Since \(X\) has dimension \(5\), the canonical map
\[
G/O\longrightarrow P_5(G/O)
\]
induces a bijection
\[
[X,G/O]\cong [X,P_5(G/O)].
\]

The above homotopy-group calculation gives a Postnikov fibration
\[
K(\mathbb Z,4)
\longrightarrow P_5(G/O)
\longrightarrow K(\mathbb Z/2,2),
\]
possibly with a nontrivial \(k\)-invariant.

Let \(f\colon X\to P_5(G/O)\) be a map.  Its composite with
\(P_5(G/O)\to K(\mathbb Z/2,2)\) represents a class in
\[
H^2(X;\mathbb Z/2)=0,
\]
and is therefore null-homotopic.  Hence \(f\) is homotopic to a map
into the fiber \(K(\mathbb Z,4)\).  Such maps are classified by
\[
H^4(X;\mathbb Z)=0.
\]
Thus \(f\) is null-homotopic, and consequently
\[
[X,G/O]=0.
\]
\end{proof}

\begin{lemma}\label{lem:relative-normal}
There is a natural isomorphism
\[
\cN_{\partial}(X\times I)\cong\Z/2.
\]
\end{lemma}

\begin{proof}
Collapsing \(X\times\partial I\) gives
\[
(X\times I)/(X\times\partial I)
\simeq
\Susp X\vee S^1.
\]
Since \(G/O\) is simply connected, the \(S^1\)-summand makes no
contribution.  Hence
\[
\cN_{\partial}(X\times I)
\cong[\Susp X,G/O].
\]

The suspension \(\Susp X\) is simply connected, and suspension shifts
reduced homology:
\[
\widetilde H_i(\Susp X;\Z)
\cong
\widetilde H_{i-1}(X;\Z).
\]
Consequently, \(\Susp X\) is a simply connected integral homology
six-sphere. The Whitehead theorem gives
\[
\Susp X\simeq S^6.
\]
Therefore
\[
[\Susp X,G/O]
\cong
\pi_6(G/O)
\cong
\Z/2.
\]
\end{proof}

By \cref{lem:normal-invariants-zero,lem:relative-normal},
\eqref{eq:surgery-sequence} reduces to
\begin{equation}\label{eq:reduced-surgery-sequence}
\Z/2
\xrightarrow{\theta}
L_6^s(\Z\pi)
\xrightarrow{\omega}
\cS^s_{\DIFF}(X)
\longrightarrow0.
\end{equation}

In particular, the Wall action is transitive, and its stabilizer at
the base point is finite.

\subsection{The representation ring}

\subsubsection{The negative representation space}

Let \(R_{\C}(\pi)\) be the complex representation ring of \(\pi\), and
let
\[
\overline{\phantom{x}}\colon R_{\C}(\pi)\longrightarrow R_{\C}(\pi)
\]
be the involution induced by complex conjugation, equivalently by
taking dual representations for unitary finite-dimensional
representations.  Define
\[
R^-_{\Q}(\pi)
=
\left\{
u\in R_{\C}(\pi)\otimes\Q
\;\middle|\;
\overline{u}=-u
\right\}.
\]
If \(\chi\) is an irreducible character not isomorphic to its complex
conjugate, then
\[
\chi-\overline{\chi}\in R^-_{\Q}(\pi)
\]
is nonzero.

We use Wall's \emph{simple} decoration throughout.  Define
\(\sigma_\pi\) to be Wall's \(G\)-signature homomorphism, with the
normalization appearing in the rho-variation formula
\cite[Equation~(1.2)]{CrowleyMacko}:
\[
\sigma_\pi\colon L_6^s(\Z\pi)\longrightarrow R^-_{\Q}(\pi)
\]
for the finite group \(\pi\).  The rational multisignature theorem
states that its rationalization is an isomorphism:
\begin{equation}\label{eq:multisignature-rational}
\sigma_\pi\otimes\Q\colon
L_6^s(\Z\pi)\otimes\Q
\xrightarrow{\;\cong\;}
R^-_{\Q}(\pi).
\end{equation}
This is the \(4k+2\)-dimensional part of the finite-group
multisignature calculation; see
\cite[Theorems~13A.2--13A.4]{Wall} and Petrie~\cite{Petrie}.

\subsubsection{A complex-type representation of
\texorpdfstring{\(\SL(2,7)\)}{SL(2,7)}}

The quotient $\PSL(2,7)$
has a pair of nonisomorphic complex-conjugate irreducible
representations of dimension three.  Their characters are
traditionally denoted by $\chi_3$ and $\overline{\chi}_3.$
For example, on the two conjugacy classes of elements of order seven,
the corresponding character values involve
\[
\frac{-1+\sqrt{-7}}{2}
\qquad\text{and}\qquad
\frac{-1-\sqrt{-7}}{2}.
\]
See the character table in~\cite{Atlas}.

Pulling these representations back along
\[
\SL(2,7)\longrightarrow\PSL(2,7)
\]
gives a complex-conjugate pair of irreducible representations of
\(\pi=\SL(2,7)\).  Hence $\chi_3-\overline{\chi}_3\neq0$ in $R^-_{\Q}(\pi),$
and therefore \(R^-_{\Q}(\pi)\neq0\).

\begin{proposition}\label{prop:L6-infinite}
The group
\[
L_6^s(\Z[\SL(2,7)])
\]
contains an element \(a\) of infinite order such that
\[
\sigma_\pi(a)\neq0.
\]
\end{proposition}

\begin{proof}
By \eqref{eq:multisignature-rational}, the nonzero vector
\(\chi_3-\overline{\chi}_3\) is represented by an element of
\(L_6^s(\Z\pi)\otimes\Q\).  After multiplying by a positive integer to
clear denominators, we obtain an element
\(a\in L_6^s(\Z\pi)\) with \(\sigma_\pi(a)\neq0\).  Such an element
cannot be torsion, since its image lies in the \(\Q\)-vector space
\(R^-_{\Q}(\pi)\).  Thus \(a\) has infinite order.
\end{proof}

\subsection{Infinitely many homology five-spheres}

Choose \(a\in L_6^s(\Z\pi)\) as in
\cref{prop:L6-infinite}.  For \(k\in\Z\), let
\[
[X_k,f_k]
=
\omega(ka)
\in
\cS^s_{\DIFF}(X),
\]
where
\[
f_k\colon X_k\longrightarrow X
\]
is a simple homotopy equivalence.

Since \(f_k\) is a homotopy equivalence,
\[
H_*(X_k;\Z)\cong H_*(X;\Z)\cong H_*(S^5;\Z),
\]
and
\[
\pi_1(X_k)\cong \pi.
\]
Thus every \(X_k\) is a smooth integral homology five-sphere.

\begin{proposition}\label{prop:distinct-marked}
The elements
\[
[X_k,f_k]\in\cS^s_{\DIFF}(X),
\qquad k\in\Z,
\]
are pairwise distinct.
\end{proposition}

\begin{proof}
Suppose
\[
[X_k,f_k]=[X_\ell,f_\ell].
\]
Exactness of \eqref{eq:reduced-surgery-sequence} implies
\[
(k-\ell)a\in\im\bigl(
\theta\colon\Z/2\to L_6^s(\Z\pi)
\bigr).
\]
By exactness, this image is the stabilizer of the base point under the
Wall action.  It is therefore a subgroup of \(L_6^s(\Z\pi)\), and it
has at most two elements because its source has two elements.  Hence
every element of the image is killed by two, so
\(2(k-\ell)a=0\).  Since \(a\) has infinite order, \(k=\ell\).
\end{proof}

Distinct marked structures do not automatically imply distinct
underlying manifolds, since the same manifold can admit several
different markings.  We use the rho invariant to remove this
ambiguity.

\subsubsection{The rho invariant}

Let
\[
V_{\Q}(\pi):=R_{\C}(\pi)\otimes\Q,
\qquad
R^-_{\Q}(\pi)
:=
\left\{
u\in V_{\Q}(\pi)
\;\middle|\;
\overline{u}=-u
\right\}.
\]
We define
\[
\widehat R_{\Q}(\pi)
=
V_{\Q}(\pi)/\langle\operatorname{reg}_{\pi}\rangle
\]
and
\[
\widehat R^-_{\Q}(\pi)
=
\left\{
\widehat u\in\widehat R_{\Q}(\pi)
\;\middle|\;
\overline{\widehat u}=-\widehat u
\right\}.
\]

Since the regular representation is self-conjugate,
the subspace
\(\langle\operatorname{reg}_{\pi}\rangle\subset V_{\Q}(\pi)\)
is preserved by complex conjugation. Hence complex conjugation descends
to an involution on the quotient \(\widehat R_{\Q}(\pi)\), and the
quotient map
\[
q\colon V_{\Q}(\pi)\longrightarrow \widehat R_{\Q}(\pi)
\]
restricts to a map
\[
q^-\colon R^-_{\Q}(\pi)\longrightarrow
\widehat R^-_{\Q}(\pi).
\]

In fact, this map is an isomorphism.

To prove injectivity, suppose that \(u\in R^-_{\Q}(\pi)\) satisfies
\(q^-(u)=0\). Then
\[
u=a\,\operatorname{reg}_{\pi}
\]
for some \(a\in\Q\). Since the regular representation is
self-conjugate, we have
\[
\overline u
=
a\,\overline{\operatorname{reg}_{\pi}}
=
a\,\operatorname{reg}_{\pi}
=
u.
\]
On the other hand, since \(u\in R^-_{\Q}(\pi)\), we also have
\[
\overline u=-u.
\]
Consequently \(u=-u\), and hence \(2u=0\). Since we are working over
\(\Q\), it follows that \(u=0\). Thus \(q^-\) is injective.

To prove surjectivity, let
\[
\widehat v\in\widehat R^-_{\Q}(\pi).
\]
Choose a representative \(v\in V_{\Q}(\pi)\) such that
\(q(v)=\widehat v\). The condition
\(\overline{\widehat v}=-\widehat v\) means that
\[
q(\overline v+v)=0.
\]
Therefore, for some \(a\in\Q\),
\[
\overline v+v
=
a\,\operatorname{reg}_{\pi}.
\]
Set
\[
u
:=
v-\frac{a}{2}\operatorname{reg}_{\pi}.
\]
Because the regular representation is self-conjugate, we obtain
\[
\begin{aligned}
\overline u
&=
\overline v-\frac{a}{2}\operatorname{reg}_{\pi} \\
&=
a\,\operatorname{reg}_{\pi}-v
-\frac{a}{2}\operatorname{reg}_{\pi} \\
&=
-v+\frac{a}{2}\operatorname{reg}_{\pi} \\
&=
-u.
\end{aligned}
\]
Thus \(u\in R^-_{\Q}(\pi)\). Moreover,
\[
q(u)=q(v)=\widehat v,
\]
since \(u-v\) is a rational multiple of the regular representation.
Hence \(q^-\) is surjective.

We therefore have a natural isomorphism
\[
R^-_{\Q}(\pi)
\overset{\cong}{\longrightarrow}
\widehat R^-_{\Q}(\pi).
\]
Equivalently, quotienting by the self-conjugate regular representation
does not change the anti-invariant part of the rational representation
ring. We denote the image of \(\sigma_{\pi}(x)\) in the quotient by
\(\widehat{\sigma}_{\pi}(x)\).

Let
\[
\lambda\colon X\longrightarrow B\pi
\]
classify the universal covering of \(X\), and define
\[
\lambda_k=\lambda\circ f_k\colon X_k\longrightarrow B\pi.
\]
For an oriented closed five-manifold \(M\) equipped with a reference
map \(\mu\colon M\to B\pi\), the Wall--Atiyah--Singer rho invariant
takes values in
\[
\rho(M,\mu)\in\widehat R^-_{\Q}(\pi).
\]
Indeed, for a closed oriented manifold of dimension \(2d-1\), the rho
invariant takes values in the \((-1)^d\)-eigenspace of the rational
complex representation ring modulo the regular representation; here
\(5=2\cdot3-1\), so \((-1)^d=-1\)
(cf. \cite[Section~2.1, Definition~2.2]{CrowleyMacko}).

We use the topological rho invariant. On smooth manifolds, this invariant agrees with the classical
Atiyah--Singer invariant
\cite[Section~7]{AtiyahSingerIII}.
Its surgery-theoretic extension is invariant under oriented
topological \(h\)-cobordism over \(B\pi\); see
Petrie~\cite{Petrie}, Wall~\cite[Chapter~14]{Wall}, and
Crowley--Macko~\cite{CrowleyMacko}.  It is also natural
under automorphisms of \(\pi\): if
\(\alpha\in\Aut(\pi)\), then
\[
\rho(M,B\alpha\circ\mu)
=
\alpha^*\rho(M,\mu).
\]
Here \(\alpha^*\) denotes pullback of characters by precomposition
with \(\alpha\).  Under the opposite convention \(\alpha^{-1}\)
appears instead; this gives the same automorphism orbit used below.
Reversing the orientation of \(M\) changes the sign of the invariant.
These properties, together with compatibility with Wall realization,
are proved in the surgery-theoretic formulation of the rho invariant;
see Petrie~\cite{Petrie}, Wall~\cite[Chapter 14]{Wall}, and
Crowley--Macko~\cite{CrowleyMacko}.

Define the reduced rho invariant of a marked structure by
\[
\rhot([f])
=
\rho(M,\lambda\circ f)-\rho(X,\lambda),
\qquad
[f\colon M\to X]\in\cS^s_{\DIFF}(X).
\]
With the normalization fixed above, the Wall-realization
formula is
\begin{equation}\label{eq:rho-wall-realization}
\rhot\bigl(\omega(x)\bigr)
=
\widehat\sigma_\pi(x),
\qquad
x\in L_6^s(\Z\pi).
\end{equation}
Consequently,
\begin{equation}\label{eq:rho-variation}
\rho(X_k,\lambda_k)
=
\rho(X,\lambda)
+
k\,\widehat\sigma_\pi(a).
\end{equation}
Since \(\sigma_\pi(a)\neq0\) and the negative eigenspace injects into
the reduced representation ring, the vector
\(\widehat\sigma_\pi(a)\) is nonzero.  Thus the values in
\eqref{eq:rho-variation} are pairwise distinct.

\begin{proposition}\label{prop:nonhomeomorphic-spheres}
There is an infinite subset \(J\subset\Z\) such that the manifolds
\[
\{X_j\}_{j\in J}
\]
are pairwise nonhomeomorphic and pairwise not topologically
\(h\)-cobordant.
\end{proposition}

\begin{proof}
The finite group \(\Aut(\pi)\times\{\pm1\}\) acts on
\(\widehat R^-_{\Q}(\pi)\), where \(\Aut(\pi)\) acts by pullback of
representations and the second factor acts by multiplication by
\(-1\).  Consider the affine sequence
\[
v_k
=
\rho(X,\lambda)+k\,\widehat\sigma_\pi(a),
\qquad k\in\Z.
\]
It is infinite because \(\widehat\sigma_\pi(a)\neq0\).  Every
\(\Aut(\pi)\times\{\pm1\}\)-orbit is finite, so the set
\(\{v_k\mid k\in\Z\}\) meets infinitely many distinct orbits.  Choose
an infinite subset \(J\subset\Z\) containing at most one index from
each such orbit.

Suppose that a topological \(h\)-cobordism $(W;X_k,X_\ell)$ exists. Let
\[
i_k\colon X_k\longrightarrow W,
\qquad
i_\ell\colon X_\ell\longrightarrow W
\]
denote the boundary inclusions. Since \(i_k\) is a homotopy
equivalence and \(X_k\) is orientable, the identity
\[
i_k^*w_1(W)=w_1(X_k)=0
\]
implies that \(W\) is orientable.

Choose a homotopy inverse
\[
r_k\colon W\longrightarrow X_k
\]
to \(i_k\), and define
\[
\nu=\lambda_k\circ r_k\colon W\longrightarrow B\pi.
\]
Then
\[
\nu\circ i_k\simeq\lambda_k.
\]
The map \(\nu\circ i_\ell\colon X_\ell\to B\pi\) also induces an
isomorphism on fundamental groups.  Hence, after using the marking
of \(X_\ell\) to identify its fundamental group with \(\pi\), there
is an automorphism \(\alpha\in\Aut(\pi)\) such that
\[
\nu\circ i_\ell\simeq B\alpha\circ\lambda_\ell.
\]
Topological \(h\)-cobordism invariance, naturality under \(\alpha\),
and comparison of the given orientations with the boundary
orientations therefore give
\[
v_k=\varepsilon\,\alpha^*v_\ell
\]
for some \(\varepsilon\in\{\pm1\}\).
Thus \(v_k\) and \(v_\ell\) belong to the same
\(\operatorname{Aut}(\pi)\times\{\pm1\}\)-orbit. By the choice of \(J\),
this implies \(k=\ell\). Hence the manifolds \(X_j\), \(j\in J\), are
pairwise not topologically \(h\)-cobordant. In particular, they are
pairwise nonhomeomorphic.
\end{proof}

\subsection{Contractible fillings}

We now apply Kervaire's theorem.

\begin{theorem}[{\cite[Theorem 3]{KervaireHomologySpheres}}]\label{thm:kervaire-filling}
Let \(\Sigma^n\) be a smooth oriented homology sphere with \(n\geq5\).
There exists a unique smooth homotopy sphere \(\Theta^n\) such that
\[
\Sigma^n\#\Theta^n
\]
bounds a compact contractible smooth \((n+1)\)-manifold.
\end{theorem}

In dimensions five and six, Kervaire--Milnor proved
(cf. \cite{KervaireMilnor})
\[
\Theta_5=0,
\qquad
\Theta_6=0.
\]
Therefore no homotopy-sphere correction is needed in either boundary
dimension.

\begin{corollary}\label{cor:all-bound}
Every smooth oriented homology five-sphere bounds a compact
contractible smooth six-manifold, and every smooth oriented homology
six-sphere bounds a compact contractible smooth seven-manifold.
\end{corollary}

\begin{corollary}[Dimension-six nonrigidity]
\label{cor:dimension-six}
There is an infinite family of compact contractible smooth
six-manifolds whose interiors are mutually properly homotopy equivalent
and pairwise nonhomeomorphic.
\end{corollary}

\begin{proof}
For each \(j\in J\), choose a compact contractible smooth
six-manifold \(C_j\) with \(\partial C_j=X_j\), using
\cref{cor:all-bound}.  The manifolds \(X_j\) are all
simple-homotopy equivalent to \(X\), and
\cref{prop:nonhomeomorphic-spheres} shows that they are pairwise not
topologically \(h\)-cobordant.  The conclusion now follows directly
from \cref{thm:boundary-to-interior-transfer}.
\end{proof}

\section{A general criterion in even dimensions}
\label{sec:general-even-criterion}

We now isolate the mechanism behind the six-dimensional construction.
Unlike the argument in the preceding sections, the general proof does
not require a calculation of the normal invariant groups of the
chosen homology sphere.

\subsection{Reduced multisignatures}

Let \(N=2d\geq6\), and let \(\pi\) be a finite group.  For
\(\varepsilon\in\{+1,-1\}\), define
\[
R^\varepsilon_{\Q}(\pi)
=
\left\{
u\in R_{\C}(\pi)\otimes\Q
\;\middle|\;
\overline u=\varepsilon u
\right\}.
\]
Since the regular representation is self-conjugate, complex
conjugation descends to
\[
\widehat R_{\Q}(\pi)
=
\bigl(R_{\C}(\pi)\otimes\Q\bigr)/
\Q\operatorname{reg}_\pi.
\]
Write
\[
\widehat R^\varepsilon_{\Q}(\pi)
=
\left\{
\widehat u\in\widehat R_{\Q}(\pi)
\;\middle|\;
\overline{\widehat u}=\varepsilon\widehat u
\right\}.
\]

Because \(\operatorname{reg}_\pi\) belongs to the positive eigenspace
and \(2\) is invertible over \(\Q\), the quotient map induces canonical
identifications
\[
\widehat R^-_{\Q}(\pi)\cong R^-_{\Q}(\pi),
\qquad
\widehat R^+_{\Q}(\pi)
\cong
R^+_{\Q}(\pi)/\Q\operatorname{reg}_\pi.
\]
We use these identifications without further comment.

\begin{theorem}[Even-dimensional criterion]
\label{thm:general-criterion}
Let \(N=2d\geq6\), and let \(\pi\) be a finite superperfect group.
Suppose that
\[
\widehat R^{(-1)^d}_{\Q}(\pi)\neq0.
\]
Then there exists an infinite family of compact contractible smooth
\(N\)-manifolds \(\{C_j\}_{j\in J}\) such that their interiors
\(M_j=\Int(C_j)\) are all properly homotopy equivalent but pairwise
nonhomeomorphic.  The boundaries \(\partial C_j\) may be chosen to be
homotopy equivalent integral homology \((N-1)\)-spheres with
fundamental group \(\pi\), and they are pairwise not topologically
\(h\)-cobordant.
\end{theorem}

All manifolds in this section are oriented, so the orientation
character is trivial and \(\Z\pi\) carries the involution
\(g\mapsto g^{-1}\).  Define
\(\sigma_{d,\pi}\) to be Wall's unreduced \(G\)-signature
homomorphism:
\[
\sigma_{d,\pi}\colon
L_{2d}^s(\Z\pi)
\longrightarrow
R^{(-1)^d}_{\Q}(\pi).
\]
Its kernel and cokernel are torsion, so rationalization gives an
isomorphism
\[
\sigma_{d,\pi}\otimes\Q\colon
L_{2d}^s(\Z\pi)\otimes\Q
\xrightarrow{\;\cong\;}
R^{(-1)^d}_{\Q}(\pi)
\]
\cite[Theorems~13A.2--13A.4]{Wall}; see also
Petrie~\cite{Petrie}.  Composing with the quotient by the regular
representation gives the reduced multisignature
\[
\widehat\sigma_{d,\pi}\colon
L_{2d}^s(\Z\pi)
\longrightarrow
\widehat R^{(-1)^d}_{\Q}(\pi).
\]
We use Wall's convention for which this reduced map agrees with the
map denoted \(\sigma_\lambda\) in
\cite[Equation~(1.2)]{CrowleyMacko}.

We recall the compatibility between the rho invariant and Wall
realization.  Let \(M\) be a closed oriented manifold of dimension
\(2d-1\), let \(\mu\colon M\to B\pi\) be a reference map, and put
\(\Gamma=\pi_1(M)\).  For a marked structure
\([h\colon P\to M]\in\cS^s_{\TOP}(M)\), define
\[
\rhot_\mu([h])
=
\rho(P,\mu\circ h)-\rho(M,\mu).
\]
This invariant takes values in
\(\widehat R^{(-1)^d}_{\Q}(\pi)\)
\cite[Definitions~2.2 and~2.5]{CrowleyMacko}.

The homomorphism \(\mu_*\colon\Gamma\to\pi\) determines a reduced
multisignature homomorphism
\[
\widehat\sigma_\mu\colon
L_{2d}^s(\Z\Gamma)
\longrightarrow
\widehat R^{(-1)^d}_{\Q}(\pi).
\]
With this normalization, the Wall-realization formula is
\begin{equation}\label{eq:general-rho-wall}
\rhot_\mu\bigl([h]+x\bigr)
=
\rhot_\mu([h])
+
\widehat\sigma_\mu(x),
\qquad
x\in L_{2d}^s(\Z\Gamma).
\end{equation}
If \(\mu_*\colon\Gamma\to\pi\) is the chosen identification of
fundamental groups, then
\[
\widehat\sigma_\mu
=\widehat\sigma_{d,\pi}.
\]
The formula and its application to the
simple structure set follow from Petrie~\cite{Petrie} and
Crowley--Macko~\cite[Equation~(1.2) and Remark~1.7]{CrowleyMacko};
see also Wall~\cite[Chapter~14]{Wall}.

\begin{lemma}\label{lem:criterion-wall-element}
If \(\widehat R^{(-1)^d}_{\Q}(\pi)\neq0\), then there exists an
infinite-order element \(a\in L_{2d}^s(\Z\pi)\) such that
\(\widehat\sigma_{d,\pi}(a)\neq0\).
\end{lemma}

\begin{proof}
Choose \(0\neq u\in\widehat R^{(-1)^d}_{\Q}(\pi)\), and choose a lift
\(\widetilde u\in R^{(-1)^d}_{\Q}(\pi)\).  Since
\(\sigma_{d,\pi}\otimes\Q\) is surjective, there exists
\(x\in L_{2d}^s(\Z\pi)\otimes\Q\) such that
\[
(\sigma_{d,\pi}\otimes\Q)(x)=\widetilde u.
\]
After multiplying by a sufficiently large positive integer \(m\), we
obtain \(a\in L_{2d}^s(\Z\pi)\) satisfying \(a\otimes1=mx\).  It
follows that
\[
\widehat\sigma_{d,\pi}(a)=m\,u\neq0.
\]
Since a torsion element has zero image in a rational vector space,
\(a\) has infinite order.
\end{proof}

\subsection{Proof of the general criterion}

\begin{proof}[Proof of \cref{thm:general-criterion}]
Let \(N=2d\geq6\) and \(n=N-1=2d-1\).  Since \(\pi\) is finite and
superperfect, it is finitely presented and satisfies
\(H_1(\pi;\Z)=H_2(\pi;\Z)=0\).  Kervaire's realization theorem
therefore gives a closed oriented smooth integral homology
\(n\)-sphere \(X\) such that
\[
\pi_1(X)\cong\pi
\]
\cite[Theorem~1]{KervaireHomologySpheres}.  Fix such an identification,
and let \(\lambda\colon X\to B\pi\) classify the universal covering.

Choose \(a\in L_N^s(\Z\pi)\) as in
\cref{lem:criterion-wall-element}.  Wall realization defines marked
structures
\[
[X_k,f_k]
=
\omega(ka)
\in
\cS^s_{\DIFF}(X),
\qquad
k\in\Z,
\]
where \(f_k\colon X_k\to X\) is a simple homotopy equivalence.
Choose the orientation of \(X_k\) so that \(f_k\) has degree \(+1\).
In particular, each \(X_k\) is a smooth integral homology \(n\)-sphere,
and \(f_k\) identifies
\(\pi_1(X_k)\cong\pi_1(X)\cong\pi\).  Put
\[
\lambda_k=\lambda\circ f_k\colon X_k\longrightarrow B\pi.
\]

The smooth normal cobordism realizing \(ka\), regarded as a
topological normal cobordism, has the same surgery obstruction and
the same \(G\)-signature defect.  Therefore
\cite[Equation~(1.2) and Remark~1.7]{CrowleyMacko}, applied with
reference map \(\lambda\), gives
\begin{equation}\label{eq:criterion-rho-sequence}
\rho(X_k,\lambda_k)
=
\rho(X,\lambda)
+
k\,\widehat\sigma_{d,\pi}(a).
\end{equation}
Set \(v_k=\rho(X_k,\lambda_k)\).  Since
\(\widehat\sigma_{d,\pi}(a)\neq0\), the elements \(v_k\), \(k\in\Z\),
are pairwise distinct.

Since \(\pi\) is finite, the group
\(\Gamma_\pi=\Aut(\pi)\times\{\pm1\}\) is finite.  It acts on
\(\widehat R^{(-1)^d}_{\Q}(\pi)\) by
\[
(\alpha,\varepsilon)\cdot v
=
\varepsilon\,\alpha^*v.
\]
Every \(\Gamma_\pi\)-orbit is finite, whereas
\(\{v_k\mid k\in\Z\}\) is infinite.  Hence this set meets infinitely
many distinct \(\Gamma_\pi\)-orbits.  Choose an infinite subset
\(J\subset\Z\) containing at most one index from each such orbit.

We claim that the manifolds \(\{X_j\}_{j\in J}\) are pairwise not
topologically \(h\)-cobordant.  Suppose that a topological
\(h\)-cobordism \((W;X_i,X_j)\) exists, and let
\[
\iota_i\colon X_i\longrightarrow W,
\qquad
\iota_j\colon X_j\longrightarrow W
\]
denote the boundary inclusions.

Since \(\iota_i\) is a homotopy equivalence and \(X_i\) is orientable,
the identity
\[
\iota_i^*w_1(W)=w_1(X_i)=0
\]
implies \(w_1(W)=0\).  Thus \(W\) is orientable.  Choose a homotopy
inverse \(r_i\colon W\to X_i\) to \(\iota_i\), and define
\[
\nu=\lambda_i\circ r_i\colon W\longrightarrow B\pi.
\]
Then \(\nu\circ\iota_i\simeq\lambda_i\).

The map \(\nu\circ\iota_j\colon X_j\to B\pi\) also induces an
isomorphism on fundamental groups.  After using the marking \(f_j\)
to identify \(\pi_1(X_j)\) with \(\pi\), there is an automorphism
\(\alpha\in\Aut(\pi)\) such that
\[
\nu\circ\iota_j
\simeq
B\alpha\circ\lambda_j.
\]
If basepoints are suppressed, \(\alpha\) is determined only up to an
inner automorphism.  This ambiguity is harmless because inner
automorphisms act trivially on the complex representation ring.

Topological \(h\)-cobordism invariance of the rho invariant over
\(B\pi\), naturality under \(\alpha\), and comparison of the fixed
orientations with the boundary orientations induced by \(W\) give
\[
v_i
=
\varepsilon\,\alpha^*v_j
\]
for some \(\varepsilon\in\{\pm1\}\).  Thus \(v_i\) and \(v_j\) belong
to the same \(\Gamma_\pi\)-orbit.  By the choice of \(J\), this forces
\(i=j\).  Hence the manifolds \(X_j\), \(j\in J\), are pairwise not
topologically \(h\)-cobordant.

For each \(j\in J\), Kervaire's filling theorem supplies a smooth
homotopy \(n\)-sphere \(\Sigma_j\) such that
\[
Y_j:=X_j\#\Sigma_j
\]
bounds a compact contractible smooth \(N\)-manifold \(C_j\)
\cite[Theorem~3]{KervaireHomologySpheres}.  Since \(n\geq5\), the
generalized Poincar\'e theorem implies that every smooth homotopy
\(n\)-sphere is homeomorphic to \(S^n\)
\cite{SmalePoincare}.  Therefore
\[
Y_j
=
X_j\#\Sigma_j
\cong_{\TOP}
X_j\#S^n
\cong_{\TOP}
X_j.
\]
It follows that the \(Y_j\), \(j\in J\), are pairwise not
topologically \(h\)-cobordant.  Moreover, each \(Y_j\) has fundamental
group \(\pi\), and each \(Y_j\) is homotopy equivalent to \(X\).

Applying \cref{thm:boundary-to-interior-transfer} to the family
\(\{C_j\}_{j\in J}\) completes the proof.
\end{proof}

\subsection{The two congruence classes and
\texorpdfstring{\(\SL(2,7)\)}{SL(2,7)}}

The criterion has different representation-theoretic meanings in the
two congruence classes of even dimensions.

\begin{proposition}\label{prop:sl27-all-even}
Let \(\pi=\SL(2,7)\).  For every integer \(d\geq3\),
\[
\widehat R^{(-1)^d}_{\Q}(\pi)\neq0.
\]
More precisely:
\begin{enumerate}[label=\textup{(\roman*)}]
\item if \(d\) is odd, then \(R^-_{\Q}(\pi)\neq0\);
\item if \(d\) is even, then
      \(R^+_{\Q}(\pi)/\Q\operatorname{reg}_\pi\neq0\).
\end{enumerate}
\end{proposition}

\begin{proof}
Suppose first that \(d\) is odd.  The character table of
\(\PSL(2,7)\) contains a nonisomorphic complex-conjugate pair
\(\chi_3,\overline{\chi}_3\) of irreducible characters of degree
\(3\).  On the two conjugacy classes of elements of order \(7\), their
values include
\[
\frac{-1+i\sqrt7}{2}
\qquad\text{and}\qquad
\frac{-1-i\sqrt7}{2};
\]
see \cite{Atlas}.  In particular,
\(\chi_3\neq\overline{\chi}_3\).

Inflation along the quotient homomorphism
\(\SL(2,7)\to\PSL(2,7)\) preserves irreducibility and
nonisomorphism.  Thus the inflated characters, denoted by the same
symbols, satisfy
\[
0\neq
\chi_3-\overline{\chi}_3
\in
R^-_{\Q}(\pi).
\]
This proves part \textup{(i)}.

Now suppose that \(d\) is even.  Let \(\mathbf1_\pi\) denote the
trivial character.  Since it is self-conjugate,
\(\mathbf1_\pi\in R^+_{\Q}(\pi)\).  Its class in
\(R^+_{\Q}(\pi)/\Q\operatorname{reg}_\pi\) is nonzero.  Indeed, if
\(\mathbf1_\pi=q\operatorname{reg}_\pi\) for some \(q\in\Q\), then
evaluation at any nonidentity element \(g\in\pi\) would give
\[
1
=
\mathbf1_\pi(g)
=
q\operatorname{reg}_\pi(g)
=
0,
\]
which is impossible.  This proves part \textup{(ii)}.
\end{proof}

\begin{proof}[Proof of \cref{thm:even-main}]
Let \(N=2d\geq6\) be even.  Apply
\cref{thm:general-criterion} to the finite superperfect group
\(\pi=\SL(2,7)\).  The group is superperfect by
\cref{lem:sl27-superperfect}, and the required reduced representation
space is nonzero by \cref{prop:sl27-all-even}.
\end{proof}

\begin{remark}\label{rem:mod-four-difference}
Suppose first that \(N\equiv2\pmod4\).  Then \(d=N/2\) is odd, and
the construction is detected by the negative representation space
\(R^-_{\Q}(\pi)\).  This space is nonzero precisely when complex
conjugation acts nontrivially on the rationalized complex
representation ring.  Thus a non-self-conjugate pair of complex
irreducible representations is required.  For
\(\pi=\SL(2,7)\), the class
\(\chi_3-\overline{\chi}_3\) provides the required nonzero direction.

Suppose instead that \(N\equiv0\pmod4\).  Then \(d\) is even, and
the construction is detected by
\[
R^+_{\Q}(\pi)/\Q\operatorname{reg}_\pi.
\]
For every nontrivial finite group \(\pi\), the trivial character
defines a nonzero element of this quotient.  Hence, in dimensions
divisible by four, the representation-theoretic hypothesis of
\cref{thm:general-criterion} is automatic for every nontrivial finite
superperfect group.
\end{remark}

\section{A general odd-dimensional nonrigidity criterion}
\label{sec:odd-general-criterion}

Let
\[
N=2d+1\geq7,
\qquad
m=N-1=2d\geq6.
\]
The purpose of this section is to separate the geometric and
algebraic inputs of the odd-dimensional construction from the
particular group used later.  The criterion below applies whenever a
group has finite outer automorphism group, admits a suitable
antisimple homology-sphere realization, and has a nonzero reduced
odd-dimensional rational \(L\)-class.

\begin{definition}\label{def:antisimple}
Let \(Y^m\) be a closed oriented even-dimensional manifold and put
\(\pi_Y=\pi_1(Y)\).  We call \(Y\) \emph{antisimple} if
\(C_*(\widetilde Y)\) is chain-homotopy equivalent over
\(\Z\pi_Y\) to a finite complex of finitely generated projective
modules \(P_*\) satisfying
\[
P_{m/2}=0.
\]
This is the algebraic hypothesis needed for the construction of the
absolute higher rho invariant \cite{WeinbergerRho,DFW}.
\end{definition}

The adjective in \cref{def:antisimple} is used in the algebraic sense
of Weinberger and Dranishnikov--Ferry--Weinberger.  We will say
\emph{algebraically antisimple} when a distinction is needed.
Hausmann's handle-theoretic notion of a \(k\)-antisimple manifold,
used only in \cref{sec:dimension-seven}, is stronger and will always
be identified explicitly.

Let \(L^*(\Z\pi)\) denote symmetric \(L\)-theory.  We first
record the comparison needed to use topological bordisms in the
absolute higher-rho construction while retaining the smooth bordism
calculation used later.

\begin{lemma}[Rational comparison of smooth and topological oriented bordism]
\label{lem:MSO-MSTOP-comparison}
For every connected CW complex \(Z\), the forgetful homomorphism
\begin{equation}\label{eq:MSO-MSTOP-isomorphism}
\Omega_n^{SO}(Z)\otimes\Q
\longrightarrow
\Omega_n^{\STOP}(Z)\otimes\Q
\end{equation}
is an isomorphism.  Moreover, the symmetric-signature transformations
fit into a commutative diagram
\begin{equation}\label{eq:MSO-MSTOP-signature-diagram}
\begin{tikzcd}[column sep=large]
\Omega_n^{SO}(Z)\otimes\Q
  \arrow[r,"\sigma^*_{SO}"]
  \arrow[d,"\cong"']
& L^n(\Z[\pi_1(Z)])\otimes\Q
  \arrow[d,equal] \\
\Omega_n^{\STOP}(Z)\otimes\Q
  \arrow[r,"\sigma^*_{\STOP}"']
& L^n(\Z[\pi_1(Z)])\otimes\Q.
\end{tikzcd}
\end{equation}
Consequently, the smooth and topological symmetric-signature images
in rational symmetric \(L\)-theory coincide.
\end{lemma}

\begin{proof}
Stable smoothing theory implies that the inclusion
\[
BSO\longrightarrow B\STOP
\]
is a rational homotopy equivalence.  Equivalently, the induced map of
oriented Thom spectra
\[
MSO\longrightarrow M\STOP
\]
is a rational equivalence; see
\cite{KirbySiebenmann,MadsenMilgram}.  After smashing with
\(Z_+\), taking stable homotopy groups, and tensoring with \(\Q\), one
obtains \eqref{eq:MSO-MSTOP-isomorphism}.

The symmetric signature is defined from the symmetric Poincar\'e chain
complex of the universal cover.  Forgetting a smooth structure does
not change this chain-level symmetric Poincar\'e complex.  Naturality
of the symmetric signature therefore gives
\eqref{eq:MSO-MSTOP-signature-diagram}, and the final assertion
follows.
\end{proof}

Define
\begin{equation}\label{eq:odd-reduced-target}
\cR_N(\pi)
:=
\operatorname{coker}\!\left(
\Omega_N^{\STOP}(B\pi)\otimes\Q
\xrightarrow{\ \sigma^*\ }
L^N(\Z\pi)\otimes\Q
\right).
\end{equation}
By \cref{lem:MSO-MSTOP-comparison}, this is canonically isomorphic to
the cokernel obtained by replacing \(\Omega_N^{\STOP}\) with
\(\Omega_N^{SO}\).  In particular, all smooth-bordism calculations
below compute the target in \eqref{eq:odd-reduced-target}.
Write
\[
q_\pi\colon
L^N(\Z\pi)\otimes\Q
\longrightarrow
\cR_N(\pi)
\]
for the quotient map.  For \(b\in L_N^s(\Z\pi)\), set
\begin{equation}\label{eq:odd-reduced-symmetrization}
\operatorname{sym}_\pi(b)
:=
q_\pi\bigl(\operatorname{sym}(b)\otimes1\bigr),
\end{equation}
where
\(
\operatorname{sym}\colon L_N^s(\Z\pi)\to L^N(\Z\pi)
\)
denotes the composite
\begin{equation}\label{eq:odd-decoration-comparison}
\begin{aligned}
L_N^s(\Z\pi)
=L_N^{\langle2\rangle}(\Z\pi)
&\longrightarrow
L_N^h(\Z\pi)
=L_N^{\langle1\rangle}(\Z\pi)\\
&\longrightarrow
L_N^p(\Z\pi)
=L_N^{\langle0\rangle}(\Z\pi)
\xrightarrow{\ 1+T\ }
L^N(\Z\pi).
\end{aligned}
\end{equation}
Here \(L_N^p=L_N^{\langle0\rangle}\) is the projective quadratic
decoration and \(1+T\) is quadratic-to-symmetric symmetrization.
The relative terms in the first two comparison maps are
Tate cohomology groups of algebraic \(K\)-groups and are therefore
killed after inverting \(2\); quadratic-to-symmetric
symmetrization is likewise an isomorphism after inverting \(2\).
Thus every map in \eqref{eq:odd-decoration-comparison} becomes an
isomorphism after tensoring with \(\Q\); see
\cite{RanickiALT,RanickiAG}.  Consequently, if
\(\cR_N(\pi)\neq0\), rational surjectivity followed by clearing
denominators gives an integral Wall class with nonzero image under
\(\operatorname{sym}_\pi\).

We next make explicit the algebraic truncation used below.  Put
\(d=m/2\).  If \(P_*\) is a finite projective symmetric Poincar\'e
chain model for \(C_*(\widetilde Y)\) with \(P_d=0\), define its
brutal truncation by
\begin{equation}\label{eq:odd-brutal-truncation}
\bigl(P^{<d}\bigr)_j
:=
\begin{cases}
P_j,&j<d,\\
0,&j\geq d.
\end{cases}
\end{equation}
The vanishing of \(P_d\) makes the projection
\[
r_P:P_*\longrightarrow P^{<d}_*
\]
a chain retraction.  The transported symmetric Poincar\'e structure
on \(P_*\) equips \(r_P\) with the structure of an
\(N\)-dimensional symmetric algebraic Poincar\'e pair whose boundary
is \(P_*\).  We denote this pair by \(\mathcal T(P)\).  The
construction is natural up to symmetric algebraic bordism under chain
homotopy equivalences of antisimple projective models
\cite[Proposition~2.26]{DFW}.

\begin{lemma}[Algebraic cancellation across an \(h\)-cobordism]
\label{lem:absolute-rho-hcob-cancellation}
Let \((U;Y_0,Y_1)\) be an oriented topological \(h\)-cobordism over
\(B\pi\), with
\(\partial U=(-Y_0)\sqcup Y_1\), between closed oriented antisimple
\(m\)-manifolds.  Choose projective models \(P_i\) and identify them
by transporting \(P_0\) across the chain equivalence induced by
\(U\).  If
\(\mathcal C(U;P_0,P_1)\) denotes the resulting symmetric algebraic
Poincar\'e triad, then
\begin{equation}\label{eq:absolute-rho-relative-cancellation}
\mathcal T(P_1)
\cup_{P_1}
\bigl(-\mathcal C(U;P_0,P_1)\bigr)
\ \sim_{\mathrm{sym}}\
\mathcal T(P_0)
\end{equation}
as symmetric Poincar\'e pairs with boundary \(P_0\).

Consequently, if \(V_0\) is a topological null-bordism of \(Y_0\)
over \(B\pi\) and \(V_1=V_0\cup_{Y_0}U\), then, for these compatible
choices,
\begin{equation}\label{eq:absolute-rho-hcob-upstairs}
\bigl[
\mathcal T(P_1)\cup-\mathcal C(V_1,Y_1)
\bigr]
=
\bigl[
\mathcal T(P_0)\cup-\mathcal C(V_0,Y_0)
\bigr]
\quad\text{in }L^N(\Z\pi).
\end{equation}
\end{lemma}

\begin{proof}
Choose a homotopy inverse \(r_0:U\to Y_0\) of the inclusion of
\(Y_0\), and put \(f=r_0|_{Y_1}:Y_1\to Y_0\).  After choosing the
standard homotopies to the reference maps, \(f\) is an
orientation-preserving homotopy equivalence over \(B\pi\).
After replacing the boundary complexes by the chosen projective
models, write \(K\) for the middle complex of the symmetric
Poincar\'e triad of \(U\), and write
\[
j_i:P_i\longrightarrow K
\qquad(i=0,1)
\]
for its two boundary maps.  Both lifted relative chain complexes
\[
C_*(\widetilde U,\widetilde Y_0)
\qquad\text{and}\qquad
C_*(\widetilde U,\widetilde Y_1)
\]
are contractible, so both \(j_i\) are chain equivalences.  Choose a
chain-homotopy inverse to \(j_0\); its composite with \(j_1\) is the
chain equivalence induced by \(f\).  The standard homotopies exhibit
the diagram
\[
P_0\xrightarrow{j_0}K\xleftarrow{j_1}P_1
\]
as chain-homotopy equivalent, relative to its two ends, to the
algebraic mapping-cylinder diagram of \(f\).  Transport the full
symmetric Poincar\'e-triad structure across this equivalence.  The
relative symmetric construction and its homotopy invariance are given
in \cite[Propositions~6.2--6.3]{RanickiSurgeryII}.  The precise algebraic
gluing and Poincar\'e-triad statements used here are
\cite[\S1.7 and \S2.1, especially Proposition~2.1.1]{RanickiExact}; compare also
\cite[Definitions~1.6--1.7 and Proposition~1.13]{RanickiALT}.

Use the transported model as \(P_1\), and write
\(f_P:P_1\to P_0\) for the induced chain equivalence.  Since
\((P_i)_d=0\), a chain-homotopy inverse of \(f_P\), together with its
inverse homotopies, restricts below degree \(d\).  Hence
\(f_P^{<d}:P_1^{<d}\to P_0^{<d}\) is a chain equivalence and
\[
\begin{tikzcd}
P_1 \arrow[r,"f_P"] \arrow[d,"r_{P_1}"']
& P_0 \arrow[d,"r_{P_0}"]\\
P_1^{<d} \arrow[r,"f_P^{<d}"']
& P_0^{<d}
\end{tikzcd}
\]
is a chain-homotopy equivalence of truncation pairs.  With the
transported symmetric structures it gives an equivalence
\[
F:\mathcal T(P_1)\simeq\mathcal T(P_0),
\]
well defined up to symmetric bordism by the model-independence
argument in \cite[Proposition~2.26]{DFW}.

Orient the algebraic mapping cylinder of \(F\) as an
\((N+1)\)-dimensional symmetric Poincar\'e triad \(\mathcal B_F\)
whose faces are \(\mathcal T(P_1)\), \(-\mathcal T(P_0)\), and
\(-\mathcal M(f_P)\).  On the other hand, take the oppositely
oriented algebraic mapping-cylinder bordism of the chain-homotopy
equivalence, relative to both corner complexes, between
\(\mathcal C(U;P_0,P_1)\) and \(\mathcal M(f_P)\).  Its horizontal
faces are \(-\mathcal C(U;P_0,P_1)\) and
\(\mathcal M(f_P)\), while its vertical product faces, with their
induced corner orientations, are \(\mathcal M(1_{P_0})\) and
\(-\mathcal M(1_{P_1})\).

Glue the two bordisms along
\(-\mathcal M(f_P)\sqcup\mathcal M(f_P)\).  The face
\(-\mathcal M(1_{P_1})\) is the product collar between the two
\(P_1\)-corners.  Rounding that corner absorbs the collar into
\(\mathcal T(P_1)\cup_{P_1}-\mathcal C(U;P_0,P_1)\); it is not glued
to a second copy of \(\mathcal M(1_{P_1})\).  Algebraic
Poincar\'e-triad gluing therefore gives a bordism \(\mathcal B(U)\)
whose four remaining \(N\)-dimensional faces are
\[
\mathcal T(P_1),
\qquad
-\mathcal C(U;P_0,P_1),
\qquad
-\mathcal T(P_0),
\qquad
\mathcal M(1_{P_0}).
\]
Equivalently, separating its horizontal and vertical boundaries gives
\[
\partial_{\mathrm h}\mathcal B(U)
=
\left(
\mathcal T(P_1)
\cup_{P_1}-\mathcal C(U;P_0,P_1)
\right)
\sqcup\bigl(-\mathcal T(P_0)\bigr),
\qquad
\partial_{\mathrm v}\mathcal B(U)
=
\mathcal M(1_{P_0}).
\]
The vertical face is a product cylinder, so \(\mathcal B(U)\) is a
symmetric bordism of pairs relative to \(P_0\).  This
proves \eqref{eq:absolute-rho-relative-cancellation}.

No simple-contractibility assertion has entered the argument.  If the
Whitehead torsion of \(f\) is nonzero, it remains encoded in the
mapping-cylinder face \(\mathcal M(f_P)\); the same non-simple face
occurs with the opposite boundary orientation when the two triads are
glued and becomes an internal face.  Hence projective symmetric \(L\)-theory
allows an arbitrary \(h\)-cobordism here, not just an
\(s\)-cobordism.

Finally, algebraic gluing gives
\[
\mathcal C(V_1,Y_1)
\simeq
\mathcal C(V_0,Y_0)
\cup_{P_0}
\mathcal C(U;P_0,P_1).
\]
Associativity of algebraic union, followed by
\eqref{eq:absolute-rho-relative-cancellation}, gives
\begin{align*}
\mathcal T(P_1)\cup-\mathcal C(V_1,Y_1)
&\sim_{\mathrm{sym}}
\left(
\mathcal T(P_1)\cup-\mathcal C(U;P_0,P_1)
\right)
\cup-\mathcal C(V_0,Y_0)\\
&\sim_{\mathrm{sym}}
\mathcal T(P_0)\cup-\mathcal C(V_0,Y_0).
\end{align*}
Passing to the closed symmetric Poincar\'e cobordism class proves
\eqref{eq:absolute-rho-hcob-upstairs}.
\end{proof}

\begin{proposition}[Absolute higher rho package]
\label{prop:absolute-rho-package}
Let \(Y^m\) be a closed oriented antisimple manifold, let
\[
\mu\colon Y\longrightarrow B\pi
\]
classify its universal cover, and suppose that \((Y,\mu)\) is
oriented topologically null-bordant over \(B\pi\).  Then there is a class
\[
\rhoabs(Y,\mu)\in\cR_N(\pi)
\]
with the following properties.
\begin{enumerate}[label=\textup{(\roman*)}]
\item If \(\alpha\in\Aut(\pi)\), then
\begin{equation}\label{eq:absolute-rho-naturality}
\rhoabs(Y,B\alpha\circ\mu)
=
\alpha_*\rhoabs(Y,\mu).
\end{equation}
The action factors through \(\Out(\pi)\).

\item Orientation reversal changes the sign:
\begin{equation}\label{eq:absolute-rho-orientation}
\rhoabs(-Y,\mu)=-\rhoabs(Y,\mu).
\end{equation}

\item If an oriented topological \(h\)-cobordism over \(B\pi\) joins
\((Y_0,\mu_0)\) to \((Y_1,\mu_1)\), and one end satisfies the
hypotheses above, then so does the other and
\[
\rhoabs(Y_1,\mu_1)=\rhoabs(Y_0,\mu_0).
\]

\item If \([Y_b,f_b]=\omega(b)\) is obtained from \(Y\) by Wall
realization of \(b\in L_N^s(\Z\pi)\), and
\(\mu_b=\mu\circ f_b\), then \(Y_b\) is antisimple and
null-bordant over \(B\pi\), and
\begin{equation}\label{eq:odd-rho-wall-package}
\rhoabs(Y_b,\mu_b)-\rhoabs(Y,\mu)
=
\operatorname{sym}_\pi(b)
\quad\text{in }\cR_N(\pi).
\end{equation}
\end{enumerate}
\end{proposition}

\begin{proof}
Choose a finite projective chain model \(P_*\) for
\(C_*(\widetilde Y)\) with \(P_{m/2}=0\), and transport the symmetric
Poincar\'e structure of \(Y\) to \(P_*\).  The retraction
\eqref{eq:odd-brutal-truncation} gives the symmetric algebraic
Poincar\'e pair \(\mathcal T(P)\) with algebraic boundary the symmetric
complex of \(Y\).

Choose a compact oriented topological null-bordism \(V\) of \(Y\)
together with an extension \(\overline\mu:V\to B\pi\) of \(\mu\).
Its existence is precisely the null-bordism hypothesis.  The use of a
topological, rather than necessarily smooth, null-bordism agrees with
the definition of the target \eqref{eq:odd-reduced-target}.  Let
\(\mathcal C(V,Y)\) denote the symmetric algebraic Poincar\'e pair of
\((V,Y)\), with coefficients induced by \(\overline\mu\).  Define
\begin{equation}\label{eq:absolute-rho-representative}
\rho(P,V)
:=
\bigl[
\mathcal T(P)\cup_{\partial}-\mathcal C(V,Y)
\bigr]
\in L^N(\Z\pi)\otimes\Q
\end{equation}
and
\begin{equation}\label{eq:absolute-rho-definition}
\rhoabs(Y,\mu)
:=q_\pi\bigl(\rho(P,V)\bigr).
\end{equation}
This is Weinberger's absolute higher rho invariant.  The construction
originates in \cite{WeinbergerRho}; the algebraic truncation and the
Wall-realization calculation used here are given in the proof of
\cite[Proposition~2.26]{DFW}.  Its
\cite[Remark~2.27]{DFW} concerns the separate extension in which an
actual null-bordism is not assumed and is not used here.

A chain equivalence between two antisimple projective models identifies
the corresponding truncation pairs up to symmetric algebraic bordism.
If \(V_0\) and \(V_1\) are two topological null-bordisms, the
algebraic gluing theorem gives
\begin{equation}\label{eq:absolute-rho-bordism-change}
\rho(P,V_0)-\rho(P,V_1)
=
\sigma^*\bigl(
[V_0\cup_Y(-V_1)\longrightarrow B\pi]
\bigr)
\end{equation}
in \(L^N(\Z\pi)\otimes\Q\).  The closed manifold on the right
represents an element of \(\Omega_N^{\STOP}(B\pi)\).  Hence this
difference vanishes after passage to \(\cR_N(\pi)\), so
\eqref{eq:absolute-rho-definition} is independent of all choices.  If
the chosen null-bordisms are smooth, the same conclusion follows from
\cref{lem:MSO-MSTOP-comparison}.

Change of coefficients along an automorphism of \(\pi\) is compatible
with truncation, algebraic gluing, and \(q_\pi\), proving
\eqref{eq:absolute-rho-naturality}.  Inner automorphisms induce
canonically chain-isomorphic coefficient systems, so the action
factors through \(\Out(\pi)\).  Reversing the orientation negates the
symmetric structures in \eqref{eq:absolute-rho-representative}, which
proves \eqref{eq:absolute-rho-orientation}.

Now let \((U;Y_0,Y_1)\) be an oriented topological \(h\)-cobordism
over \(B\pi\), with reference map \(\nu:U\to B\pi\).  Choose a
compact oriented topological null-bordism \(V_0\) of \((Y_0,\mu_0)\)
and put
\[
V_1=V_0\cup_{Y_0}U.
\]
Then \(V_1\) is a compact oriented topological null-bordism of
\((Y_1,\mu_1)\).  Choose an antisimple projective model \(P_0\) for
\(Y_0\).  Transporting it across either boundary equivalence gives an
antisimple model \(P_1\) for \(Y_1\).  By
\cref{lem:absolute-rho-hcob-cancellation},
\[
\rho(P_1,V_1)=\rho(P_0,V_0)
\quad\text{in }L^N(\Z\pi)\otimes\Q.
\]
Applying \(q_\pi\) proves topological \(h\)-cobordism invariance.
The topological null-bordism \(V_1\) is admissible because
\eqref{eq:odd-reduced-target} is defined using topological oriented
bordism.  By
\cref{lem:MSO-MSTOP-comparison}, this target agrees with the smooth
bordism target computed in the subsequent sections.

Finally, let a normal cobordism
\[
(T_b;Y_b,Y)\longrightarrow Y\times I
\]
realize \(b\in L_N^s(\Z\pi)\).  Orient and order its boundary so that
 \(\partial T_b=Y_b\sqcup(-Y)\), the quadratic surgery kernel
represents \(b\), and the symmetric signature defect
``\(Y_b\)-end minus \(Y\)-end'' is \(+\operatorname{sym}(b)\).
Glue its \(Y\)-boundary to \(V\), obtaining a null-bordism \(V_b\) of
\((Y_b,\mu\circ f_b)\).  Transport the projective model \(P\) along
the simple homotopy equivalence \(f_b\); this also shows that \(Y_b\)
is antisimple.  The symmetrization of the quadratic surgery kernel is
the symmetric-signature defect of the normal cobordism.  Hence the
algebraic gluing formula gives
\begin{equation}\label{eq:absolute-rho-wall-upstairs}
\rho(P_b,V_b)-\rho(P,V)
=
\operatorname{sym}(b)\otimes1
\quad\text{in }L^N(\Z\pi)\otimes\Q.
\end{equation}
This is the Wall-realization calculation in the final paragraph of
the proof of \cite[Proposition~2.26]{DFW}.  Applying \(q_\pi\) proves
\eqref{eq:odd-rho-wall-package}.
\end{proof}

\begin{theorem}[Odd-dimensional nonrigidity criterion]
\label{thm:odd-criterion}
Let \(N=2d+1\geq7\), put \(m=N-1\), and let \(\pi\) be a finitely
presented group.  Assume the following.
\begin{enumerate}[label=\textup{(\roman*)}]
\item \(\Out(\pi)\) is finite.

\item There are a smooth oriented integral homology \(m\)-sphere
\(X\), a compact smooth oriented \(N\)-manifold \(V\), and a map
\[
\overline\lambda\colon V\longrightarrow B\pi
\]
such that \(\partial V=X\), the restriction
\(\lambda=\overline\lambda|_X\) classifies the universal cover of
\(X\), the inclusion \(X\hookrightarrow V\) induces an isomorphism
on fundamental groups, and \(X\) is antisimple.

\item The reduced rational target is nonzero:
\[
\cR_N(\pi)\neq0.
\]
\end{enumerate}
Then there is an infinite family of compact contractible smooth
\(N\)-manifolds \(\{C_j\}_{j\in J}\) such that their interiors are
all properly homotopy equivalent but pairwise nonhomeomorphic.  Their
boundaries may be chosen to be homotopy equivalent integral homology
\(m\)-spheres that are pairwise not topologically
\(h\)-cobordant.
\end{theorem}

\begin{proof}
Choose
\[
0\neq\zeta_0\in\cR_N(\pi).
\]
Since \(q_\pi\) is surjective, choose
\[
\eta\in L^N(\Z\pi)\otimes\Q
\qquad\text{with}\qquad
q_\pi(\eta)=\zeta_0.
\]
Rational symmetrization is an isomorphism, so there exists
\[
x\in L_N^s(\Z\pi)\otimes\Q
\]
whose symmetric image is \(\eta\).  Multiplying by a sufficiently
large positive integer clears denominators and gives an integral
class
\[
a\in L_N^s(\Z\pi)
\]
with
\[
\zeta:=\operatorname{sym}_\pi(a)\neq0.
\]
In particular, \(a\) has infinite order.

Apply smooth Wall realization to the base manifold \(X\).  For
\(k\in\Z\), write
\[
[X_k,f_k]=\omega(ka),
\qquad
f_k\colon X_k\longrightarrow X.
\]
Each \(f_k\) is an orientation-preserving simple homotopy equivalence.
Consequently, every \(X_k\) is a smooth integral homology
\(m\)-sphere with fundamental group \(\pi\), all the \(X_k\) are
homotopy equivalent, and each \(X_k\) is antisimple.  Put
\[
\lambda_k=\lambda\circ f_k\colon X_k\longrightarrow B\pi.
\]
The normal cobordism realizing \(ka\), glued to \(V\), gives a
null-bordism of \((X_k,\lambda_k)\) over \(B\pi\).  Therefore
\cref{prop:absolute-rho-package} applies and gives
\begin{equation}\label{eq:odd-rho-variation}
\rhoabs(X_k,\lambda_k)
=
\rhoabs(X,\lambda)+k\zeta
\quad\text{in }\cR_N(\pi).
\end{equation}
Since \(\zeta\neq0\), these values are pairwise distinct.

The finite group
\[
\Out(\pi)\times\{\pm1\}
\]
acts linearly on \(\cR_N(\pi)\), using naturality for the first factor
and orientation reversal for the second.  Every orbit is finite,
whereas the affine set in \eqref{eq:odd-rho-variation} is infinite.
Choose an infinite set \(J\subset\Z\) containing at most one index
from each orbit.

We claim that \(\{X_j\}_{j\in J}\) are pairwise not topologically
\(h\)-cobordant.  Suppose that \((W;X_i,X_j)\) is a topological
\(h\)-cobordism.  Since one boundary inclusion is a homotopy
equivalence and the boundary is orientable, \(W\) is orientable.
Choose a homotopy inverse
\[
r_i\colon W\longrightarrow X_i
\]
to the inclusion of \(X_i\), and define
\[
\nu=\lambda_i\circ r_i\colon W\longrightarrow B\pi.
\]
The restriction to \(X_j\) induces an isomorphism on fundamental
groups.  Relative to the marking \(f_j\), it is homotopic to
\(B\alpha\circ\lambda_j\) for some \(\alpha\in\Aut(\pi)\), unique
up to an inner automorphism.  Topological \(h\)-cobordism invariance,
naturality, and the orientation rule imply
\[
\rhoabs(X_i,\lambda_i)
=
\varepsilon\,\alpha_*\rhoabs(X_j,\lambda_j)
\]
for some \(\varepsilon\in\{\pm1\}\).  Thus the two values lie in the
same \(\Out(\pi)\times\{\pm1\}\)-orbit.  By the choice of \(J\),
this forces \(i=j\).

For each \(j\in J\), Kervaire's filling theorem
\cite[Theorem~3]{KervaireHomologySpheres} supplies a smooth homotopy
\(m\)-sphere \(\Theta_j\) such that
\[
Y_j:=X_j\#\Theta_j
\]
bounds a compact contractible smooth \(N\)-manifold \(C_j\).  Since
\(m\geq5\), the generalized Poincar\'e theorem gives
\(\Theta_j\cong_{\TOP}S^m\) \cite{SmalePoincare}.  Hence
\[
Y_j\cong_{\TOP}X_j.
\]
The manifolds \(Y_j\) are therefore pairwise not topologically
\(h\)-cobordant and are all homotopy equivalent.

The conclusion follows from
\cref{thm:boundary-to-interior-transfer}, applied to
\(\{C_j\}_{j\in J}\).
\end{proof}

\begin{remark}\label{rem:odd-criterion-orbits}
The finiteness of \(\Out(\pi)\) is a convenient sufficient condition.
The proof only requires that the affine set
\[
\{\rhoabs(X,\lambda)+k\zeta\mid k\in\Z\}
\]
meet infinitely many \(\Out(\pi)\times\{\pm1\}\)-orbits.
\end{remark}
\section{The product group \texorpdfstring{\(\Gamma\times\SL(2,7)\)}{Gamma x SL(2,7)}}
\label{sec:odd-product-group}

We now verify the hypotheses of \cref{thm:odd-criterion} for every
odd integer \(N\geq9\).  Write
\[
N=2d+1,
\qquad
m=N-1=2d\geq8.
\]
The same group will work in every such dimension.

Choose a closed oriented hyperbolic integral homology three-sphere
\(H\), and set
\[
\Gamma=\pi_1(H).
\]
Such manifolds exist by Myers' homology-cobordism construction
\cite{MyersHyperbolic}, or see \Cref{prop:existence_of_hyperbolic_3_sphere}.

\begin{proposition}\label{prop:existence_of_hyperbolic_3_sphere}
There exists a closed oriented hyperbolic integral homology
$3$-sphere.
\end{proposition}

\begin{proof}
Let $K=4_1\subset S^3$ be the figure-eight knot, and let
\[
M=(S^3\setminus \nu K)(1,5)=S^3_{1/5}(K),
\]
where the filling slope is $\mu+5\lambda$, with $\mu$ and $\lambda$
denoting the meridian and preferred longitude of $K$, respectively.
By Thurston's explicit analysis of Dehn surgeries on the figure-eight
knot, the $(1,5)$-filling is hyperbolic
\cite[Theorem~4.7]{ThurstonGT3M}. Hence $M$ is a closed oriented
hyperbolic $3$-manifold.

On the other hand,
\[
H_1(S^3\setminus \nu K;\mathbb Z)
 \cong \mathbb Z\langle[\mu]\rangle,
\qquad
[\lambda]=0.
\]
Dehn filling along $\mu+5\lambda$ imposes the relation
$[\mu]+5[\lambda]=0$, and therefore $H_1(M;\mathbb Z)=0$. Poincaré
duality and the universal coefficient theorem then give
$H_2(M;\mathbb Z)=0$, while
$H_0(M;\mathbb Z)\cong H_3(M;\mathbb Z)\cong\mathbb Z$. Thus
\[
H_*(M;\mathbb Z)\cong H_*(S^3;\mathbb Z),
\]
so $M$ is an integral homology $3$-sphere.
\end{proof}

Since \(H\) is a closed hyperbolic
three-manifold, it is aspherical.  Therefore
\[
H_i(B\Gamma;\Z)\cong H_i(H;\Z),
\]
and hence
\begin{equation}\label{eq:Gamma-homology}
H_1(\Gamma;\Z)=H_2(\Gamma;\Z)=0,
\qquad
H_3(B\Gamma;\Z)\cong\Z,
\qquad
H_i(B\Gamma;\Z)=0\quad(i>3).
\end{equation}
In particular, \(\Gamma\) is finitely presented, torsion-free,
centerless, and superperfect.

Let
\[
F=\SL(2,7),
\qquad
\pi=\Gamma\times F.
\]
The group \(F\) is finite and superperfect by
\cref{lem:sl27-superperfect}.  The integral K\"unneth theorem and
\eqref{eq:Gamma-homology} give
\[
H_1(\pi;\Z)=H_2(\pi;\Z)=0,
\]
so \(\pi\) is a finitely presented infinite superperfect group.

\begin{lemma}\label{lem:odd-out-finite}
The group \(\Out(\pi)\) is finite.
\end{lemma}

\begin{proof}
Every finite subgroup of \(\pi\) projects trivially to the torsion-free
group \(\Gamma\), and is therefore contained in \(F\).  Thus \(F\)
is the unique maximal finite normal subgroup of \(\pi\), hence is
characteristic.  Every automorphism of \(\pi\) consequently induces
a pair of automorphisms of \(F\) and \(\pi/F\cong\Gamma\).

The kernel of the induced map
\[
\Out(\Gamma\times F)
\longrightarrow
\Out(\Gamma)\times\Out(F)
\]
is represented by central cross-homomorphisms.  The possible maps
\(F\to Z(\Gamma)\) vanish because \(Z(\Gamma)=1\), and the possible
maps \(\Gamma\to Z(F)\) vanish because \(\Gamma\) is perfect.  Hence
the displayed map is injective.  Mostow rigidity gives finiteness of
\(\Out(\Gamma)\) \cite{Mostow}, while \(\Out(F)\) is finite because
\(F\) is finite.  Therefore \(\Out(\pi)\) is finite.
\end{proof}

\section{Verification of the reduced rational \texorpdfstring{\(L\)}{L}-theory hypothesis}
\label{sec:odd-L-verification}

The reduced target \(\cR_N(\pi)\) and the map
\(\operatorname{sym}_\pi\) were defined in
\cref{sec:odd-general-criterion}.  For the product group
\(\pi=\Gamma\times F\), we now calculate this target explicitly.
Let
\[
R_{\Q}(F)=R_{\C}(F)\otimes\Q,
\qquad
R^{\pm}_{\Q}(F)
=
\{x\in R_{\Q}(F)\mid\overline{x}=\pm x\}.
\]

\begin{proposition}\label{prop:odd-rational-summand}
Relative to the product decomposition \(\pi=\Gamma\times F\), there is
an isomorphism
\begin{equation}\label{eq:odd-L-product-identification}
L^N(\Z\pi)\otimes\Q
\;\cong\;
H_3(B\Gamma;\Q)\otimes
\bigl(L^{N-3}(\Z F)\otimes\Q\bigr).
\end{equation}
Under this isomorphism, the projection of the symmetric-signature image
is
\begin{equation}\label{eq:odd-signature-image}
\operatorname{im}(\sigma^*)
=
\begin{cases}
0,
 & N\equiv1\pmod4,\\[2mm]
H_3(B\Gamma;\Q)\otimes
\Q\operatorname{reg}_F,
 & N\equiv3\pmod4.
\end{cases}
\end{equation}
Consequently,
\begin{equation}\label{eq:odd-reduced-target-explicit}
\cR_N(\pi)
\;\cong\;
H_3(B\Gamma;\Q)\otimes
\begin{cases}
R^-_{\Q}(F),&N\equiv1\pmod4,\\[2mm]
R^+_{\Q}(F)/\Q\operatorname{reg}_F,
&N\equiv3\pmod4.
\end{cases}
\end{equation}
In particular, \(\cR_N(\pi)\neq0\) for every odd \(N\geq9\).
\end{proposition}

\begin{proof}
We divide the proof into four steps.  Throughout the intermediate
calculation, \(L_*\) denotes the projective quadratic decoration
\(L_*^{\langle0\rangle}\).  The comparison
\eqref{eq:odd-decoration-comparison} identifies its rationalization
with both the simple Wall group used for realization and the symmetric
group occurring in \(\cR_N(\pi)\).

\smallskip
\noindent
\emph{Step 1: the equivariant Chern-character decomposition.}
The group \(\Gamma\) is word-hyperbolic and \(F\) is finite, so
\(\pi=\Gamma\times F\) is word-hyperbolic.  Hence the
Farrell--Jones assembly map in \(L\)-theory is an isomorphism for \(\pi\) \cite{BartelsLueck}.  The equivariant Chern character of
\cite[Theorem~178]{LueckReich} is stated for every decoration
\(\langle j\rangle\) with \(j\leq1\); we use it with \(j=0\).
It gives
\begin{equation}\label{eq:odd-chern-character-full}
L_N(\Z\pi)\otimes\Q
\cong
\bigoplus_{\substack{p+q=N\\(C)\in(\mathrm{FCY}(\pi))}}
H_p(BZ_\pi C;\Q)
\otimes_{\Q[W_\pi C]}
\bigl(\theta_C L_q(\Z C)\otimes\Q\bigr),
\end{equation}
where \((C)\) runs through the conjugacy classes of finite cyclic
subgroups, \(Z_\pi C\) is the centralizer,
\(W_\pi C=N_\pi C/Z_\pi C\), and \(\theta_C\) is the primitive
cyclic idempotent; see \cite[Theorem~178]{LueckReich}.

Because \(\Gamma\) is torsion-free, every finite cyclic subgroup of
\(\pi\) is of the form \(\{1\}\times C\) with \(C\leq F\).  Moreover,
\(\pi\)-conjugacy of such subgroups is the same as \(F\)-conjugacy,
and
\begin{equation}\label{eq:odd-centralizer-weyl}
Z_\pi C=\Gamma\times Z_FC,
\qquad
N_\pi C=\Gamma\times N_FC,
\qquad
W_\pi C=W_FC.
\end{equation}
Since \(Z_FC\) is finite,
\[
H_p(BZ_\pi C;\Q)
\cong H_p(B\Gamma;\Q).
\]
This is an isomorphism of \(\Q[W_FC]\)-modules, where the Weyl group
acts trivially on the right-hand side: conjugation by an element of
\(N_FC\) is the identity on the direct factor \(\Gamma\).  It follows
that the summand indexed by \(C\) in
\eqref{eq:odd-chern-character-full} is
\begin{equation}\label{eq:odd-separate-gamma-F}
H_p(B\Gamma;\Q)\otimes
\left(
 \Q\otimes_{\Q[W_FC]}
 \theta_C L_q(\Z C)\otimes\Q
\right).
\end{equation}

Apply the same Chern-character formula to the finite group \(F\).
For a finite centralizer \(Z_FC\), only its degree-zero rational
homology survives.  Therefore
\begin{equation}\label{eq:odd-artin-LF}
\bigoplus_{(C)\in(\mathrm{FCY}(F))}
\Q\otimes_{\Q[W_FC]}
\bigl(\theta_C L_q(\Z C)\otimes\Q\bigr)
\cong L_q(\Z F)\otimes\Q.
\end{equation}
Thus the Weyl-group coinvariants in
\eqref{eq:odd-separate-gamma-F} are not being discarded: their sum is
exactly the rational finite-group \(L\)-group.  Combining
\eqref{eq:odd-chern-character-full}--\eqref{eq:odd-artin-LF} gives
\begin{equation}\label{eq:odd-product-all-p}
L_N(\Z\pi)\otimes\Q
\cong
\bigoplus_{p+q=N}
H_p(B\Gamma;\Q)\otimes
\bigl(L_q(\Z F)\otimes\Q\bigr).
\end{equation}

\smallskip
\noindent
\emph{Step 2: only the degree-three term survives.}
By \eqref{eq:Gamma-homology}, the rational homology of \(B\Gamma\)
is nonzero only in degrees zero and three.  If \(p=0\), then
\(q=N\) is odd.  The odd-dimensional \(L\)-groups of a finite group
are torsion, and hence
\[
L_N(\Z F)\otimes\Q=0.
\]
The only nonzero term in \eqref{eq:odd-product-all-p} is consequently
\(p=3\), \(q=N-3\).  Passing by rational symmetrization to symmetric
\(L\)-theory proves \eqref{eq:odd-L-product-identification}.

\smallskip
\noindent
\emph{Step 3: the symmetric-signature indeterminacy.}
The rational Chern--Dold character for oriented bordism gives
\begin{equation}\label{eq:odd-MSO-splitting}
\Omega_N^{SO}(B\pi)\otimes\Q
\cong
\bigoplus_{r\geq0}
H_{N-4r}(B\pi;\Q)
\otimes
\bigl(\Omega_{4r}^{SO}\otimes\Q\bigr).
\end{equation}
Since \(F\) is finite, \(BF\) is rationally acyclic, so
\(H_*(B\pi;\Q)\cong H_*(B\Gamma;\Q)\).  Since \(N\) is odd, a
summand in \eqref{eq:odd-MSO-splitting} can occur only in homological
degree three.  Such a summand exists exactly when \(N\equiv3\pmod4\),
and in that case
\begin{equation}\label{eq:odd-MSO-degree-three}
\Omega_N^{SO}(B\pi)\otimes\Q
\cong
H_3(B\Gamma;\Q)
\otimes
\bigl(\Omega_{N-3}^{SO}\otimes\Q\bigr).
\end{equation}
For \(N\equiv1\pmod4\), the group on the left of
\eqref{eq:odd-MSO-splitting} is zero, proving the first line of
\eqref{eq:odd-signature-image}.

Suppose now that \(N\equiv3\pmod4\).  Under
\eqref{eq:odd-MSO-degree-three} and
\eqref{eq:odd-L-product-identification}, the required compatibility is
the commutative diagram
\begin{equation}\label{eq:odd-signature-chern-compatibility}
\begin{tikzcd}[column sep=5.2em,row sep=large]
\Omega_N^{SO}(B\Gamma\times BF)\otimes\Q
  \arrow[r,"\sigma^*"]
  \arrow[d,"\operatorname{ch}^{(3)}_{MSO}"']
& L^N(\Z[\Gamma\times F])\otimes\Q
  \arrow[d,"\operatorname{ch}^{(3)}_{L}"] \\
H_3(B\Gamma;\Q)\otimes
  \bigl(\Omega_{N-3}^{SO}\otimes\Q\bigr)
  \arrow[r,
    "\operatorname{id}\otimes
     (\operatorname{ind}_1^F\circ\sigma^*)"']
& H_3(B\Gamma;\Q)\otimes
  \bigl(L^{N-3}(\Z F)\otimes\Q\bigr).
\end{tikzcd}
\end{equation}
Here the left vertical map is the degree-three part of the rational
Chern--Dold character and the right vertical map is the degree-three
equivariant Chern-character identification constructed in Steps 1--2.
Both are isomorphisms in total degree \(N\).

We spell out why the coefficient homomorphism in the bottom row is
precisely induction from the trivial subgroup.  The symmetric
signature is the assembly of the \(L\)-homology fundamental class
obtained from the Sullivan--Ranicki orientation
\(MSO\to\mathbb L^\bullet\); see
\cite[Propositions~16.15--16.16]{RanickiALT}.  A bordism class over
\(B\pi=E\pi/\pi\) is represented equivariantly on the free
\(\pi\)-space \(E\pi\).  Therefore
\[
(E\pi)^C=\varnothing
\qquad
\text{for every nontrivial finite cyclic subgroup }C\leq\pi.
\]
Its equivariant Chern character consequently has only the component
indexed by \(C=1\).  After the \(\Gamma\)- and \(F\)-variables are
separated as in Step 1, that component is extension of scalars along
\(\Z\to\Z F\), namely \(\operatorname{ind}_1^F\).  Naturality and
compatibility with induction for the equivariant Chern character are
given in \cite[Theorems~0.3 and~6.3]{LueckChern}.

Commutativity of \eqref{eq:odd-signature-chern-compatibility} now follows
from naturality of the two Chern characters, induction along
\(1\leq F\), and the cartesian-product formula for symmetric
signatures; see \cite[Theorem~178]{LueckReich} and
\cite[Proposition~8.1(i)]{RanickiSurgeryII}.  Thus the map on
\eqref{eq:odd-MSO-degree-three} is explicitly
\begin{equation}\label{eq:odd-signature-coefficient-map}
\operatorname{id}_{H_3(B\Gamma;\Q)}
\otimes
\left(
\Omega_{N-3}^{SO}\otimes\Q
\xrightarrow{\ \sigma^*\ }
L^{N-3}(\Z)\otimes\Q
\xrightarrow{\ \operatorname{ind}_1^F\ }
L^{N-3}(\Z F)\otimes\Q
\right).
\end{equation}
Here \(\operatorname{ind}_1^F\) is extension of scalars along
\(\Z\to\Z F\).  Since \(N-3\) is divisible by four, the first map in
parentheses is the rational signature map.  Its target is
one-dimensional over \(\Q\), and a suitable product of complex
projective spaces has nonzero signature; hence the map is surjective.
Under the finite-group multisignature isomorphism,
extension of scalars sends the one-dimensional form to its tensor
product with \(\Q F\); its character is therefore
\(\operatorname{reg}_F\).  Hence the image of the coefficient map in
\eqref{eq:odd-signature-coefficient-map} is exactly
\(\Q\operatorname{reg}_F\).  This proves the second line of
\eqref{eq:odd-signature-image}.

\smallskip
\noindent
\emph{Step 4: identify the quotient and prove nonvanishing.}
Put \(q=N-3\). Recall that Wall's rational multisignature calculation gives
\begin{equation}\label{eq:odd-finite-multisignature-detailed}
L^q(\Z F)\otimes\Q
\cong
\begin{cases}
R^-_{\Q}(F),&q\equiv2\pmod4,\\[1mm]
R^+_{\Q}(F),&q\equiv0\pmod4;
\end{cases}
\end{equation}
see \cite[Chapter~13A]{Wall} and \cite{Petrie}.  Quotienting
\eqref{eq:odd-L-product-identification} by
\eqref{eq:odd-signature-image} now gives
\eqref{eq:odd-reduced-target-explicit}.

Finally, \(R^-_{\Q}(F)\neq0\): inflation of the nonisomorphic
complex-conjugate characters \(\chi_3,\overline\chi_3\) of
\(\PSL(2,7)\) gives
\[
0\neq\chi_3-\overline\chi_3\in R^-_{\Q}(F)
\]
\cite{Atlas}.  Also
\(R^+_{\Q}(F)/\Q\operatorname{reg}_F\neq0\), since the trivial
character is self-conjugate and cannot be a rational multiple of the
regular character: evaluation at any nonidentity element of \(F\)
would otherwise give \(1=0\).  This proves the final assertion.
\end{proof}

\section{Verification of the antisimple realization hypothesis}
\label{sec:odd-antisimple-verification}

The notion of an antisimple manifold was introduced in
\cref{def:antisimple}.  We now construct the realization required by
\cref{thm:odd-criterion}.

\begin{lemma}\label{lem:odd-kervaire-antisimple}
There exist a smooth oriented integral homology \(m\)-sphere \(X\), a
compact smooth oriented \((m+1)\)-manifold \(V\), and a classifying map
\(\lambda:X\to B\pi\) such that:
\begin{enumerate}[label=\textup{(\roman*)}]
\item \(\partial V=X\), and \(\lambda\) extends to a map
      \(\overline\lambda:V\to B\pi\);
\item \(\pi_1(X)\to\pi_1(V)\) is an isomorphism and both groups are
      identified with \(\pi\);
\item \(V\) has a handle decomposition with handles only of indices
      \(0,1,2,3\);
\item \(X\) is antisimple.
\end{enumerate}
More precisely, let \(C^h_*\) be the finite free
\(\Z\pi\)-chain complex determined by the handle decomposition of
\(V\).  Then \(C^h_i=0\) for
\(i\notin\{0,1,2,3\}\), and \(C_*(\widetilde X)\) is
chain-homotopy equivalent to a finite free complex \(P\) satisfying
\begin{equation}\label{eq:antisimple-explicit-modules}
P_r\cong C^h_r\oplus (C^h)^{m-r},
\end{equation}
where
\[
(C^h)^j=\operatorname{Hom}_{\Z\pi}(C^h_j,\Z\pi)
\]
with the left-module structure induced by the standard involution.
Consequently,
\begin{equation}\label{eq:antisimple-support}
P_r=0
\quad\text{unless}\quad
r\in\{0,1,2,3,m-3,m-2,m-1,m\}.
\end{equation}
\end{lemma}

\begin{proof}
\emph{The Kervaire handlebody.}
Choose a finite presentation of the finitely presented superperfect
group \(\pi\).  Start with an \((m+1)\)-dimensional \(0\)-handle and
attach one \(1\)-handle for each generator.  Its boundary is a
connected sum of copies of \(S^1\times S^{m-1}\).  Attach
\(2\)-handles along disjoint framed circles representing the
relators.  The resulting boundary has fundamental group \(\pi\).
Kervaire's use of the Hopf exact sequence and the assumption
\(H_2(\pi;\Z)=0\) shows that the remaining second-homology classes can
be represented by disjoint framed embedded two-spheres.  Attaching
\(3\)-handles along these spheres kills them.  The outgoing boundary
is a smooth integral homology \(m\)-sphere \(X\) with fundamental
group \(\pi\); see the proof of
\cite[Theorem~1]{KervaireHomologySpheres}.  Denote the resulting
handlebody by \(V\).  By construction,
\[
\partial V=X,
\qquad
V\text{ has handles only in indices }0,1,2,3.
\]
This proves (iii), as well as the homology-sphere and fundamental-group
statements needed below.

\smallskip
\noindent
\emph{The fundamental group of the boundary.}
Turn the handle decomposition upside down and regard \(V\) as obtained
from \(X\times I\) by attaching the dual handles.  The dual indices are
\[
(m+1)-3,\ (m+1)-2,\ (m+1)-1,\ (m+1)-0,
\]
namely
\[
m-2,m-1,m,m+1.
\]
Since \(m\geq8\), every dual handle has index at least six.  In
particular, no relative handle of index one or two occurs, and the
inclusion
\[
X\hookrightarrow V
\]
induces an isomorphism on fundamental groups.  Choose
\(\overline\lambda:V\to B\pi\) to classify the universal cover of
\(V\), using this identification, and put
\(\lambda=\overline\lambda|_X\).  This proves (i) and (ii).

\smallskip
\noindent
\emph{The algebraic boundary, with the two chain models separated.}
Put \(R=\Z\pi\), equipped with the standard involution
\(g\mapsto g^{-1}\).  First choose a finite cellular structure on the
pair \((V,X)\) for which \(X\subset V\) is a subcomplex, and write
\[
A=C_*(\widetilde X),
\qquad
E=C_*(\widetilde V),
\qquad
D=C_*(\widetilde V,\widetilde X).
\]
These are the cellular complexes of this particular CW-pair; at this
stage no support assertion is made about \(E\).  There is a degreewise
split short exact sequence of finite free \(R\)-complexes
\begin{equation}\label{eq:antisimple-short-exact}
0\longrightarrow A
\xrightarrow{\,i\,}E
\xrightarrow{\,q\,}D
\longrightarrow0.
\end{equation}

The relative symmetric construction makes
\[
i:A\longrightarrow E
\]
an \((m+1)\)-dimensional symmetric Poincar\'e pair.  Its leading
duality map is the chain-level equivariant Poincar\'e--Lefschetz
equivalence
\[
E^{m+1-*}\xrightarrow{\simeq}D.
\]
Choose a chain-homotopy inverse
\begin{equation}\label{eq:antisimple-PL-duality}
\delta:D
\xrightarrow{\simeq}
E^{m+1-*}.
\end{equation}
For the relative symmetric construction, the algebraic Poincar\'e-pair
structure, and its invariance under equivalence of pairs, see
\cite[Propositions~6.2--6.3]{RanickiSurgeryII} and
\cite[Definitions~1.6--1.7, 1.14 and Remark~2.5]{RanickiALT}.

Now let \(C^h\) denote the finite free handle chain complex of \(V\).
The handle decomposition supplies
\[
C^h_j=0\qquad(j\notin\{0,1,2,3\}).
\]
Since \(C^h\) and \(E\) are two finite free chain models of the same
universal cover, choose chain-homotopy inverse equivalences
\[
u:C^h\xrightarrow{\simeq}E,
\qquad
v:E\xrightarrow{\simeq}C^h.
\]
Set \(i^h=vi:A\to C^h\).  The homotopy \(uvi\simeq i\), together
with \(u\), gives an equivalence from the pair \(i^h:A\to C^h\) to
the cellular pair \(i:A\to E\).  Transport the entire symmetric
Poincar\'e-pair structure along this equivalence.  Thus the low-handle
model is used only after the relative symmetric structure has been
transferred to it; it is not identified literally with the cellular
complex \(E\).

On the underlying chain complexes put
\[
\varphi=\delta q:
E\longrightarrow E^{m+1-*}.
\]
The exact sequence \eqref{eq:antisimple-short-exact} gives the
distinguished triangle
\[
A\longrightarrow E\xrightarrow{q}D\longrightarrow\Sigma A.
\]
Replacing \(D\) by the chain-equivalent complex \(E^{m+1-*}\) gives
\[
A\simeq
\Sigma^{-1}\operatorname{Cone}(\varphi).
\]

Let
\[
u^\#:E^{m+1-*}\longrightarrow(C^h)^{m+1-*}
\]
be the dual chain equivalence induced by \(u\), and define
\[
\varphi^h
=u^\#\varphi u:
C^h\longrightarrow(C^h)^{m+1-*}.
\]
Because \(u\) and \(v\) are homotopy inverses,
\[
v^\#u^\#=(uv)^\#\simeq 1.
\]
Consequently \(u\) on the source and \(v^\#\) on the target give a
homotopy-commutative comparison between the arrows \(\varphi^h\) and
\(\varphi\).  Both comparison maps are chain equivalences, so the
induced map of mapping cones is a chain equivalence.  Therefore
\begin{equation}\label{eq:antisimple-correct-cone}
A
\simeq
P:=
\Sigma^{-1}\operatorname{Cone}(\varphi^h).
\end{equation}
The transported higher symmetric data identifies this equivalence
with the algebraic-boundary equivalence of symmetric complexes.  For
antisimplicity only its underlying chain equivalence is needed.

With the conventions
\[
\operatorname{Cone}(\varphi^h)_s
=
(C^h)^{m+1-s}\oplus C^h_{s-1},
\qquad
(\Sigma^{-1}F)_r=F_{r+1},
\]
the underlying modules of \(P\) are
\[
P_r
=
(C^h)^{m-r}\oplus C^h_r
\cong
C^h_r\oplus(C^h)^{m-r}.
\]
This proves \eqref{eq:antisimple-explicit-modules}.  Since
\(C^h_j=0\) outside degrees \(0,1,2,3\), the handle summand contributes
only for \(0\leq r\leq3\), while the dual summand contributes only
for \(m-3\leq r\leq m\).  Hence
\eqref{eq:antisimple-support} follows.

Finally, \(m\geq8\) and \(m\) is even, so
\[
3<\frac m2<m-3.
\]
Thus \(P_{m/2}=0\).  The complex \(P\) is finite free, hence finite
projective, over \(\Z\pi\), and
\eqref{eq:antisimple-correct-cone} identifies it with
\(C_*(\widetilde X)\) up to chain homotopy.  Therefore \(X\) is
antisimple, proving (iv).
\end{proof}

\section{Verification of the criterion and proof of the odd-dimensional theorem}
\label{sec:odd-criterion-application}

We can now prove the uniform odd-dimensional result by checking the
three hypotheses of \cref{thm:odd-criterion}.

\begin{proof}[Proof of \cref{thm:odd-main}]
Let \(N=2d+1\geq9\) and put \(m=N-1\).  Take
\[
\pi=\Gamma\times\SL(2,7)
\]
as in \cref{sec:odd-product-group}.  The group is finitely presented,
and \(\Out(\pi)\) is finite by \cref{lem:odd-out-finite}.

The calculation in \cref{prop:odd-rational-summand} gives
\[
\cR_N(\pi)\neq0.
\]
The Kervaire construction in \cref{lem:odd-kervaire-antisimple}
produces a smooth integral homology \(m\)-sphere \(X\), a compact
smooth \(N\)-manifold \(V\) with \(\partial V=X\), and a classifying
map over \(B\pi\) such that \(X\hookrightarrow V\) induces an
isomorphism on fundamental groups and \(X\) is antisimple.  Thus all
hypotheses of \cref{thm:odd-criterion} hold.  Applying that theorem
produces the required infinite family of compact contractible smooth
\(N\)-manifolds.
\end{proof}

\begin{corollary}\label{cor:dimension-nine}
There is a proper homotopy type containing infinitely many pairwise
nonhomeomorphic smooth open contractible nine-manifolds, each of which
is the interior of a compact contractible smooth nine-manifold.
\end{corollary}

\begin{proof}
For \(N=9\), one has \(m=8\) and \(N-3=6\).  The nonzero reduced
rational class may be taken in
\[
H_3(B\Gamma;\Q)
\otimes
R^-_{\Q}(\SL(2,7)).
\]
If \(\chi_3\) and \(\overline\chi_3\) are the nonisomorphic
complex-conjugate three-dimensional characters inflated from
\(\PSL(2,7)\), then
\[
[H]\otimes(\chi_3-\overline\chi_3)\neq0.
\]
The Kervaire trace has handles of index at most three, and the
algebraic boundary of its eight-dimensional boundary is supported in
degrees
\[
0,1,2,3,5,6,7,8.
\]
Thus the middle degree four vanishes, so the hypotheses of
\cref{thm:odd-criterion} are satisfied.
\end{proof}

\begin{proof}[Proof of \cref{thm:main-classification}]
Part \textup{(i)} is \cref{thm:low-dimensional-rigidity}.
Part \textup{(ii)} follows from \cref{thm:even-main} in every even
dimension \(N\geq6\), and from \cref{thm:odd-main} in every odd
dimension \(N\geq9\).
\end{proof}

\section{The unresolved dimension seven}
\label{sec:dimension-seven}

For \(N=7\), the boundary dimension is six, and the Kervaire trace has
no middle-dimensional gap to which the absolute higher-rho construction
applies.  Thus the odd-dimensional realization established above begins
only in dimension nine.

Hausmann's Whitehead--Tate obstruction does not repair this gap for the
binary icosahedral group
\[
\Delta=\langle a,b\mid a^5=b^3=(ab)^2\rangle.
\]
For an \(r\)-dimensional referred antisimple manifold it lies in
\[
\tau(M,g)\in
\widehat H^{\,r+1}\bigl(\Z/2;\Wh(\Delta)\bigr),
\]
and \cite[Proposition~4.4]{HausmannCellDispensability} realizes suitable
classes in the kernel of the Rothenberg boundary.  The nonzero classes in
\cite[Example~4.5(3)]{HausmannCellDispensability}, however, have degree
\(2j\), so Proposition~4.4 associates them with \(r=2j-1\), not
\(r=2j\).\footnote{Example~4.5 is printed with \(r=2j\), which is
inconsistent with Proposition~4.4's degree \(r+1\); the former is
therefore an index slip.}  In particular, they give no six-dimensional
homology sphere.

Moreover, Ushitaki proved \(\operatorname{SK}_1(\Z\Delta)=0\)
\cite[Theorem~A]{UshitakiMilnor}.  For finite groups,
\(\operatorname{SK}_1\) is the torsion subgroup of the Whitehead group,
and the standard involution is the identity modulo it; see
\cite[p.~404]{UshitakiMilnor} and
\cite[Corollary~7.5]{OliverWhitehead}.  Hence \(\Wh(\Delta)\) is
torsion-free and \(*=1\), so
\[
\widehat H^{\,7}\bigl(\Z/2;\Wh(\Delta)\bigr)
=
\frac{\ker(1+*)}{\operatorname{im}(1-*)}
=
\ker\bigl(2:\Wh(\Delta)\longrightarrow\Wh(\Delta)\bigr)
=
0.
\]
Thus Hausmann's construction with \(\Delta\) does not settle dimension
seven.

\begin{problem}[Dimension-seven nonrigidity]
\label{prob:dimension-seven}
Do there exist compact contractible smooth seven-manifolds \(C_0\) and
\(C_1\) such that
\[
\Int(C_0)\simeqp\Int(C_1)
\qquad\text{but}\qquad
\Int(C_0)\not\cong_{\TOP}\Int(C_1)?
\]
\end{problem}

By \cref{cor:all-bound}, it would suffice to find homotopy-equivalent
smooth integral homology six-spheres that are not topologically
\(h\)-cobordant.

\section{Dimension five and homology four-spheres}
\label{sec:dimension-five}

Kervaire proved that every smooth integral homology four-sphere bounds a
compact contractible smooth five-manifold
\cite[Theorem~3]{KervaireHomologySpheres}.  Thus filling is not the
obstruction, and the dimension-five question is the following.

\begin{problem}[Dimension-five nonrigidity]
\label{prob:dimension-five}
Do there exist compact contractible smooth five-manifolds \(C_0\) and
\(C_1\) such that
\[
\Int(C_0)\simeqp\Int(C_1)
\qquad\text{but}\qquad
\Int(C_0)\not\cong_{\TOP}\Int(C_1)?
\]
\end{problem}

For the boundary formulation, topological and smooth \(h\)-cobordism
coincide in the relevant dimension.

\begin{lemma}[Topological versus smooth \(h\)-cobordism]
\label{lem:homology-four-sphere-smoothing}
Let \(\Sigma_0\) and \(\Sigma_1\) be closed smooth integral homology
four-spheres.  They are topologically \(h\)-cobordant if and only if
they are smoothly \(h\)-cobordant.
\end{lemma}

\begin{proof}
Only the forward implication requires proof.  For a topological
\(h\)-cobordism \((W;\Sigma_0,\Sigma_1)\), extend the boundary smoothings
over product collars.  Its relative Kirby--Siebenmann obstruction lies
in
\[
H^4(W,\partial W;\Z/2)
\cong H_1(W;\Z/2)
\cong H_1(\Sigma_0;\Z/2)
=0,
\]
by Poincar\'e--Lefschetz duality and \(\Sigma_0\simeq W\).  Hence the
boundary PL structures extend~\cite{KirbySiebenmann}.  Since \(\PL/O\)
is six-connected, relative smoothing theory smooths \(W\) fixed on the
boundary collars
\cite[Theorem~2.1]{DaherPowellSmoothing}; see also
\cite{MadsenMilgram}.  Thus \(W\) is a smooth \(h\)-cobordism.
\end{proof}

\begin{problem}[Homology four-spheres and smooth \(h\)-cobordism]
\label{prob:homology-four-sphere-hcobordism}
Must every pair of homotopy-equivalent closed smooth integral homology
four-spheres be smoothly \(h\)-cobordant?
\end{problem}

\begin{proposition}[A sufficient boundary route in dimension five]
\label{prop:homology-four-sphere-to-dimension-five}
A negative answer to
\cref{prob:homology-four-sphere-hcobordism} gives an affirmative answer
to \cref{prob:dimension-five}.
\end{proposition}

\begin{proof}
Let \(\Sigma_0\simeq\Sigma_1\) be smooth integral homology four-spheres
that are not smoothly \(h\)-cobordant, and choose Kervaire fillings
\(\partial C_i=\Sigma_i\).  By
\cref{cor:boundary-homotopy-proper},
\[
\Int(C_0)\simeqp\Int(C_1).
\]
If the interiors were homeomorphic, \cref{prop:completion} would give a
topological, hence by \cref{lem:homology-four-sphere-smoothing} smooth,
\(h\)-cobordism between their boundaries, a contradiction.
\end{proof}

\begin{remark}
\Cref{prop:homology-four-sphere-to-dimension-five} is only sufficient;
the two open problems above need not be equivalent.  In particular,
\cref{prop:completion} supplies no converse from \(h\)-cobordant
boundaries to homeomorphic interiors.

By Wall's classification, the simply connected case of
\cref{prob:homology-four-sphere-hcobordism} is affirmative
\cite[Theorem~2]{WallSimplyConnected}.  A homeomorphic pair is likewise
topologically, hence smoothly, \(h\)-cobordant.  Thus any counterexample
must be non-simply-connected and nonhomeomorphic; the Hopf exact
sequence and \(H_1(\Sigma_i;\Z)=H_2(\Sigma_i;\Z)=0\) show that its
common fundamental group is superperfect.
\end{remark}

\bibliographystyle{amsplain}
\bibliography{bib}

\end{document}